\documentclass[10pt,hyperref]{article}
\usepackage{graphicx} %% Package for Figure
\usepackage{float} %% Package for Float
\usepackage{amssymb}
\usepackage{amsmath}
\usepackage{mathtools}
\usepackage[thmmarks,amsmath]{ntheorem} %% If amsmath is applied, then amsma is necessary
\usepackage{bm} %% Bold Mathematical Symbols
\usepackage[colorlinks,linkcolor=cyan,citecolor=cyan]{hyperref}
\usepackage{extarrows}
\usepackage[hang,flushmargin]{footmisc} %% Let the footnote not indentation
\usepackage{mathrsfs} %% Swash letter
\usepackage[font=footnotesize,skip=0pt,textfont=rm,labelfont=rm]{caption,subcaption} 
\usepackage{booktabs} %% tables with three lines
\usepackage{tocloft}
{%% Environment of Proof
    \theoremstyle{nonumberplain}
    \theoremheaderfont{\bfseries}
    \theorembodyfont{\normalfont}
    \theoremsymbol{\mbox{
   $\Box$}}
    \newtheorem{proof}{
   Proof}
}

\newcommand{\f}{\mathfrak}
\newcommand{\va}{\varepsilon}

\newtheorem{thm}{Theorem}[section]

\newtheorem{lem}[thm]{Lemma}
\newtheorem{rmk}[thm]{Remark}
\newtheorem{prop}[thm]{Proposition}
{ \theoremstyle{remark} }
\graphicspath{
    {figure/}}                                 %% Path of Figures
\usepackage[a4paper]{geometry}                           %% Paper size
\begin{document}
\title
{\textbf{ On the Hydrostatic Limit of the\\
 3D Beris-Edwards
System in a Thin Strip}}
\date{}

\author{Francesco De Anna, Xingyu Li, Marius Paicu and Arghir Zarnescu}

\renewcommand{\thefootnote}{\fnsymbol{footnote}}
%\footnotetext[1]{
   %Corresponding author. }
   %\footnotetext[2]{
   %Corresponding author. }
\maketitle

%MS+++++++++++++++++++++ Abstract +++++++++++++++++++++++++
{\noindent\small{\bf Abstract:}
 In this paper we consider the 3D  co-rotational Beris-Edwards system modeling the hydrodynamic motion of nematic liquid crystals in a thin strip.  The system contains the incompressible Navier-Stokes, coupled with a parabolic system for matrix-valued functions, the $Q$-tensors. 
     
     We show that under a suitable scaling, corresponding, in the Navier-Stokes part, to the hydrostatic scaling, one obtains in the limit  a partly decoupled system. For the fluid part we obtain  the Prandtl system while for the $Q$-tensors we obtain a non-standard system, involving fluids components and  a non-standard combination of partly dissipative equations and algebraic constraints. 
     
  We prove the convergence of the rescaled system and the well-posedness of the limit in Sobolev spaces.

\vspace{
   1ex}
    \maketitle
\setcounter{tocdepth}{1}
\tableofcontents
%MS++++++++++++++++++++++++++++++ Main body ++++++++++++++++++++

\section{Introduction}
The Beris-Edwards model is a fundamental framework in the study of complex fluid systems, particularly for describing the behavior of nematic liquid crystals. Originally formulated to model the orientational dynamics in nematic phases, the Beris-Edwards system captures the coupling between the hydrodynamics of the fluid and the orientational order of the liquid crystal molecules. This coupling is achieved through the \textit{Q-tensor} representation, which serves as an order parameter describing both the local molecular alignment direction and degree of ordering.

The Q-tensor configuration space consists of all symmetric, traceless $d\times d$ matrices ($d=2$ or 3), defined as:
\[
\mathcal S_0^{(d)}=\{Q \in\mathbb R^{d\times d}: Q=Q^T, \text{tr} Q=0\}
\]
This space accounts for the anisotropic alignment of molecules, providing a richer description of the liquid crystals internal structure than the simpler Ericksen-Leslie model, which uses only a director field to represent molecular alignment. The Ericksen-Leslie system effectively captures uniform, uniaxial configurations, but the Q-tensor formulation of the Beris-Edwards model accommodates complex phenomena like variable degrees of ordering, defect formation, and biaxiality, making it suitable for studying systems where such intricacies are relevant.

Physically, the Beris-Edwards model describes key nematic behaviors such as flow alignment, tumbling, and the dynamics of defects, which can occur when molecules are constrained by geometry, such as within thin films or strip-like domains. These geometries impose spatial restrictions that lead to unique alignment behaviors, often influenced by the anisotropy in elastic constants and boundary conditions. In particular, the  Q-tensor formulation allows for anisotropic interactions between the molecular orientation and flow, effectively capturing the impact of directional viscosities and elasticity. We refer to \cite{ADL,ADL2,BE,DHW,DZ,PZ1,PZ2,SS,WXZ} for more results about Beris-Edwards system, and \cite{CW,E1,E2,E3,HJZ,JL,Leslie1,Leslie2,W} for Ericksen-Leslie system.

The governing equations of the Beris-Edwards system consist of a coupled system involving the incompressible Navier-Stokes equations for the fluid velocity 
$\textbf u$and a Q-tensor evolution equation for the orientational dynamics:
\begin{equation}
\label{eqbe1}
\begin{cases}
\partial_t\textbf{u}+(\textbf{u}\cdot\nabla)\textbf{u}+\nabla P=\varepsilon^2\nu_1\Delta\textbf{u}-\varepsilon^4\nabla\cdot(\nabla Q \odot\nabla Q+(\Delta Q)\cdot Q-Q\cdot\Delta Q)\\
\nabla\cdot\textbf u=0\\
\partial_tQ+\textbf{u}\cdot\nabla Q+Q\Omega-\Omega Q=\varepsilon^2 \nu_2\Delta Q-\left(aQ-b[Q^2-\frac1d\mbox{tr}(Q^2)I]+cQ\mbox{tr}(Q^2)\right)
\end{cases}
\end{equation}
in these equations:
\begin{itemize}
\item[$\bullet$]
$\textbf u$ denotes the fluid velocity field,
\item[$\bullet$]
$Q$ represents the Q-tensor field, encoding molecular orientation and alignment,
\item[$\bullet$]
$P$ is the pressure,
\item[$\bullet$]
$\nu_1,\nu_2$ are viscosity constants,
\item[$\bullet$]
$\Omega:=\frac{\nabla \textbf{u}-\nabla^T \textbf{u}}{2}$ represents the antisymmetric part of the velocity gradient tensor, capturing local rotational effects,
\item[$\bullet$]
 $Q\Omega-\Omega Q$ reflects how the flow field rotates and stretches the orientational order.
\end{itemize}

The $\varepsilon$ coefficient is related to a scale on which certain characteristic patterns appear, the so-called defect patterns. This is an issue of fundamental interest in the study of nematics and the way this parameter is obtained out of a suitable non-dimensionalization is explained in \cite{WXZ}. The system in \cite{WXZ} also involves a certain ``anisotropy parameter" $\xi$ which here is set to be zero, thus providing a simpler system,  an approch also used in many other mathematical works, see for instance \cite{PZ1} and the papers citing it.

The interaction between the Q-tensor and velocity field introduces significant nonlinearity. The Navier-Stokes momentum equation includes stress terms related to gradients in $Q$, while the Q-tensor evolution equation depends on the flow field, particularly through the co-rotational term 
$Q\Omega-\Omega Q$. Additionally, the incompressibility condition $\nabla\cdot u=0$ imposes further constraints, making the analysis particularly challenging in three dimensions.\\
In this study, we focus on a \textit{3D anisotropic Beris-Edwards system} in a thin strip, defined by the domain
\[
\mathcal S^\va:=\{(x,y,z): x,y\in\Omega_0, 0<z<2\pi\va\}\subset \mathbb R^3.
\]
This setup is relevant for examining the hydrostatic limit of the Beris-Edwards system, where the thin geometry amplifies anisotropic effects in the alignment and flow dynamics. We impose either \textit{periodic boundary conditions:}
\[(u,v,w, Q_{ij})|_{z=0}=(u,v,w, Q_{ij})|_{z=2\pi\va}, \quad (\partial_zu,\partial_zv,\partial_zw, \partial_zQ_{ij})|_{z=0}=(\partial_zu,\partial_zv,\partial_zw, \partial_zQ_{ij})|_{z=2\pi\va} \] or \textit{zero boundary conditions} \[(u,v,w, Q_{ij})|_{z=0}=(u,v,w, Q_{ij})|_{z=2\pi\va}=0\] 

In the context of the co-rotational Beris-Edwards system, we assume $\xi=0$, eliminating the parameter that controls the relative strength of aligning versus tumbling effects.

This paper's  primary contribution is to analyze the behavior of this anisotropic Beris-Edwards system in thin geometries, examining its convergence to the hydrostatic limit as $\va\to 0$. The complex coupling between flow and orientation fields, compounded by anisotropic effects and boundary conditions, requires careful treatment to demonstrate convergence in Sobolev spaces. Similarly as \cite{PZZ,L1}, we consider the scaled functions
\begin{equation}
\label{3eqbe2}
\textbf{u}(t,x,y,z)=\left(u(t,x,y,z), v(t,x,y,z), w(t,x,y,z)\right)=\left(u^\varepsilon(t,x,y,z/\varepsilon), v^\varepsilon(t,x,y,z/\varepsilon),\varepsilon  w^\varepsilon(t,x,y,z/\varepsilon)\right)
\end{equation}
and
\begin{equation}
\label{3eqbe3}
P(t,x,y,z)=p^\varepsilon(t,x,y,z/\varepsilon), \quad Q_{\alpha\beta}(t,x,y,z)= \left(\begin{array}{ccc} Q_{11}^\varepsilon & Q_{12}^\varepsilon & \varepsilon Q_{13}^\varepsilon\\
Q_{12}^\varepsilon & Q_{22}^\varepsilon & \varepsilon Q_{23}^\varepsilon\\
\varepsilon Q_{13}^\varepsilon& \varepsilon Q_{23}^\varepsilon & Q^\varepsilon_{33}\end{array}\right)(x,y,z/\varepsilon)
\end{equation}
We first calculate the scaled system about the velocity field $\textbf u$.  Then the scaled system of $\textbf u^\varepsilon$ becomes:
\begin{equation}
\label{eq3u}
\begin{cases}
\partial_tu^{\varepsilon}+u^{\varepsilon}\partial_xu^{\varepsilon}+v^{\varepsilon}\partial_yu^{\varepsilon}+w^{\varepsilon}\partial_zu^{\varepsilon}+\partial_xp^{\varepsilon}=\varepsilon^2\nu\partial_x^2 u^\va+\varepsilon^2\nu\partial_y^2u^\va+\nu \partial_z^2 u^\va-\varepsilon^4(\partial_xM_{11}+\partial_yM_{12}+\partial_zM_{13})\\
\partial_tv^{\varepsilon}+u^{\varepsilon}\partial_xv^{\varepsilon}+v^{\varepsilon}\partial_yv^{\varepsilon}+w^{\varepsilon}\partial_zv^{\varepsilon}+\partial_yp^{\varepsilon}=\varepsilon^2\nu\partial_x^2v^\va+\varepsilon^2\nu\partial_y^2v^\va+\nu\partial_z^2v^\va-\varepsilon^4(\partial_xM_{21}+\partial_yM_{22}+\partial_zM_{23})\\
\varepsilon\partial_tw^{\varepsilon}+\varepsilon u^{\varepsilon}\partial_xw^{\varepsilon}+\varepsilon v^{\varepsilon}\partial_yw^{\varepsilon}+\varepsilon w^{\varepsilon}\partial_zw^{\varepsilon}+\frac{\partial_zp^{\varepsilon}}{\varepsilon}=\varepsilon^3\nu\partial_x^2w^\va+\varepsilon^3\nu\partial_y^2w^\va+\varepsilon\nu\partial_z^2w^\va-\varepsilon^4(\partial_xM_{31}+\partial_yM_{32}+\partial_zM_{33})\\
\partial_xu^{\varepsilon}+\partial_yv^{\varepsilon}+\partial_zw^{\varepsilon}=0\\
\end{cases}
\end{equation} where $M$ is defined in \eqref{def:matrixM} in the Appendix. 

Furthermore, using some of the calculations in Section~\ref{apendixsec:scaling} in the Appendix, we have that  the equations for $Q^{\varepsilon}:=(Q_{11}^\va, Q_{12}^\va, Q_{12}^\va, Q_{22}^\va, Q_{33}^\va, \va Q_{13}^\va, \va Q_{31}^\va, \va Q_{23}^\va, \va Q_{32}^\va)$ can be written as
\begin{equation}
\label{Qij1}
\begin{cases}
\partial_tQ^{\varepsilon}_{11}+u^{\varepsilon}\partial_xQ^{\varepsilon}_{11}+v^{\varepsilon}\partial_yQ^{\varepsilon}_{11}+w^{\varepsilon}\partial_zQ^{\varepsilon}_{11}-\omega^\va_0Q_{12}^\va-\omega^\va_1Q^\va_{13}=\varepsilon^2 \nu_2\partial_x^2Q^{\varepsilon}_{11}+\varepsilon^2 \nu_2\partial_y^2Q^{\varepsilon}_{11}+ \nu_2\partial_z^2Q^{\varepsilon}_{11}\\
-\left(aQ^{\varepsilon}_{11}-b((Q^\va_{11})^2+(Q^\va_{12})^2+\varepsilon^2Q^2_{13,\varepsilon}-\frac{1}{3}\mbox{tr}(Q^2))+cQ_{11}^{\varepsilon}\mbox{tr}(Q^2)\right)\\
\partial_tQ^{\varepsilon}_{12}+u^{\varepsilon}\partial_xQ^{\varepsilon}_{12}+v^{\varepsilon}\partial_yQ^{\varepsilon}_{12}+w^{\varepsilon}\partial_zQ^{\varepsilon}_{12}+\frac 12\omega_0^\va(Q^\va_{11}-Q^\va_{22})-\frac 12\omega_1^\va Q_{23}^\va-\frac 12\omega^\va_2Q_{13}^\va\\
=\varepsilon^2 \nu_2\partial_x^2Q^{\varepsilon}_{12}+\varepsilon^2 \nu_2\partial_y^2Q^{\varepsilon}_{12}+ \nu_2\partial_z^2Q^{\varepsilon}_{12}-\left(aQ^{\varepsilon}_{12}-b(Q_{12}^\va(Q_{11}^\va+Q_{22}^\va)+\varepsilon^2Q_{13}^\va Q_{23}^\va)+cQ_{12}^{\varepsilon}\mbox{tr}(Q^2)\right)\\
\partial_tQ^{\varepsilon}_{22}+u^{\varepsilon}\partial_xQ^{\varepsilon}_{22}+v^{\varepsilon}\partial_yQ^{\varepsilon}_{22}+w^{\varepsilon}\partial_zQ^{\varepsilon}_{22}+\omega^\va_0Q_{12}^\va-\omega^\va_1Q_{23}^\va=\varepsilon^2 \nu_2\partial_x^2Q^{\varepsilon}_{22}+\varepsilon^2 \nu_2\partial_y^2Q^{\varepsilon}_{22}+\partial_z^2 \nu_2Q^{\varepsilon}_{22}\\
-\left(aQ^{\varepsilon}_{22}-b((Q^\va_{12})^2+(Q^\va_{22})^2+\varepsilon^2(Q^\va_{23})^2)-\frac{1}{3}\mbox{tr}(Q^2))+cQ_{22}^{\varepsilon}\mbox{tr}(Q^2)\right)\\
\varepsilon\partial_tQ^{\varepsilon}_{13}+\varepsilon u^{\varepsilon}\partial_xQ^{\varepsilon}_{13}+\varepsilon v^{\varepsilon}\partial_yQ^{\varepsilon}_{13}+\varepsilon w^{\varepsilon}\partial_zQ^{\varepsilon}_{13}-\frac {\va}2 \omega^\va_0Q_{23}^\va+\frac 1{2\va}(\omega_1^\va( Q^\va_{11}-Q^\va_{33})+\omega^\va_2 Q_{12})\\
=\varepsilon^3 \nu_2\partial_x^2Q^{\varepsilon}_{13}+\varepsilon^3 \nu_2\partial_y^2Q^{\varepsilon}_{13}+\varepsilon \nu_2\partial_z^2Q^{\varepsilon}_{13}-\varepsilon\left(aQ^{\varepsilon}_{13}-b(-Q^\varepsilon_{22}Q^\varepsilon_{13}+Q^\varepsilon_{12}Q^\varepsilon_{23})+cQ_{13}^{\varepsilon}\mbox{tr}(Q^2)\right)\\
\varepsilon\partial_tQ^{\varepsilon}_{23}+\varepsilon u^{\varepsilon}\partial_xQ^{\varepsilon}_{23}+\varepsilon v^{\varepsilon}\partial_yQ^{\varepsilon}_{23}+\varepsilon w^{\varepsilon}\partial_zQ^{\varepsilon}_{23}+\frac{\varepsilon}2\omega^\varepsilon_0Q^\varepsilon_{13}+\frac{\omega^\varepsilon_1}{2\varepsilon}Q^\varepsilon_{12}+
\frac{\omega^\varepsilon_2}{2\varepsilon}(Q_{22}^\varepsilon-Q_{33}^\varepsilon)\\
=\varepsilon^3 \nu_2\partial_x^2Q^{\varepsilon}_{23}+\varepsilon^3 \nu_2\partial_y^2Q^{\varepsilon}_{23}+\varepsilon \nu_2\partial_z^2Q^{\varepsilon}_{23}-\varepsilon\left(aQ^{\varepsilon}_{23}-b(Q^\varepsilon_{12}Q^\varepsilon_{13}-Q^\varepsilon_{11}Q^\varepsilon_{23})+cQ_{23}^{\varepsilon}\mbox{tr}(Q^2)\right).\\
\end{cases}
\end{equation} where $\omega^\varepsilon=(\omega^\varepsilon_0,\omega^\varepsilon_1,\omega^\varepsilon_2)$ is defined in \eqref{def:omegavareps}.

\bigskip
\par\noindent{\bf\large Hydrostatic limit of the system}
When studying flows across thin layers, the hydrostatic Beris-Edwards system becomes relevant, focusing on the dynamics within a stratified regime. This hydrostatic framework emphasizes variations along the vertical $z$-direction as the primary drivers of the system’s behavior, with horizontal variations playing a less dominant role. This perspective allows us to examine the influence of the fluid flow on liquid crystal orientation, as well as how the anisotropic properties of the Q-tensor affect fluid dynamics.
In the thin-layer limit, as $\va\to 0$, the domain becomes  $(x,y,z)\in \mathcal S:=\Omega_0\times [0,2\pi]$. The velocity field 
$\textbf u$ in the anisotropic Beris-Edwards system reduces to the Prandtl system, denoted by $\tilde{\textbf{u}}=(u,v,w)$, which governs the boundary layer flow dynamics:
\begin{equation}
\label{Prandtl3}
\begin{cases}
\partial_tu+u\partial_xu+v\partial_yu+w\partial_zu+\partial_xp= \nu_1\partial_z^2u\\
\partial_tv+u\partial_xv+v\partial_yv+w\partial_zv+\partial_yp= \nu_1\partial_z^2v\\
\partial_zp=0\\
\partial_xu+\partial_yv+\partial_zw=0.\\
\end{cases}
\end{equation}
Here $(u,v,w)$ represent the components of the velocity field in the Prandtl layer, and $p$ is the pressure, which is assumed to be vertically invariant, reflecting the hydrostatic approximation.
 Introducing the \textit{semi vorticity} $\omega:=\partial_xv-\partial_yu$, which measures the rotational effect of the horizontal components, and denoting the trace term  $\text{trl}(Q^2)=2(Q_{11}^2+Q_{11}Q_{22}+Q_{22}^2+Q^2_{12})$.  the limit system for the Q-tensor components 
$Q=(Q_{11},Q_{12},Q_{22})$ is given by
\begin{equation}
\label{extra1}
\begin{cases}
\partial_tQ_{11}+\textbf{u}\cdot\nabla Q_{11}
-Q_{12}\omega+Q_{13}\partial_zu= \nu_2\partial_z^2Q_{11}-\left(aQ_{11}-b(Q^2_{11}+Q^2_{12}-\frac 13\mbox{trl}(Q^2))+cQ_{11}\mbox{trl}(Q^2)\right)\\
\partial_tQ_{12}+\textbf{u}\cdot\nabla Q_{12}
+\frac 12(Q_{11}-Q_{22})\omega+Q_{23}\partial_zu+Q_{13}\partial_zv= \nu_2\partial_z^2Q_{12}-\left(aQ_{12}-bQ_{12}(Q_{11}+Q_{22})+cQ_{12}\mbox{trl}(Q^2)\right)\\
\partial_tQ_{22}+\textbf{u}\cdot\nabla Q_{22}+Q_{12}\omega+Q_{23}\partial_zv= \nu_2\partial_z^2Q_{22}-\left(aQ_{22}-b(Q^2_{12}+Q^2_{22}-\frac 13\mbox{trl}(Q^2))+cQ_{22}\mbox{trl}(Q^2)\right)\\
(2Q_{11}+Q_{22})\partial_zu+Q_{12}\partial_zv=0\\
Q_{12}\partial_zu+(Q_{11}+2Q_{22})\partial_zv=0.
\end{cases}
\end{equation}
These equations are supplemented by either \textit{periodic boundary conditions} \[(u,v,w, Q_{ij})|_{z=0}=(u,v,w, Q_{ij})|_{z=2\pi}, \quad (\partial_zu,\partial_zv,\partial_zw, \partial_zQ_{ij})|_{z=0}=(\partial_zu,\partial_zv,\partial_zw, \partial_zQ_{ij})|_{z=2\pi} \] or \textit{zero boundary conditions} \[(u,v,w, Q_{ij})|_{z=0}=(u,v,w, Q_{ij})|_{z=12\pi}=0\] 

From the fourth and fifth equations in \eqref{extra1}, either $\partial_zu=\partial_zv=0$ holds, or  or the following relationship applies:
\begin{equation}
\label{relationQ}
Q_{12}^2=(2Q_{11}+Q_{22})(Q_{11}+2Q_{22}).
\end{equation}
In this paper we primarily consider the case that $|\partial_zu|^2+|\partial_zv|^2>0$ for $z\in [0,2\pi]$, ensuring that the relation\eqref{relationQ} holds.   Additionally, if  $\partial_zu$ and $\partial_zv$ share zeros on $\Omega_0\times [0,2\pi]$, then the ratio $\theta(t,x,y,z):=\frac{\partial_zu}{\partial_zv}$ is well defined, providing a functional relationship between $(Q_{11}, Q_{12}, Q_{22})$ and the velocity gradients.

To further understand the physical implications of the hydrostatic system, we examine how the Q-tensor fields $(Q_{11}, Q_{12}, Q_{22})$ represent the components of the alignment tensor within an anisotropic setting. These components are closely related to the vertical gradients of the velocity components, $\partial_zu$ and $\partial_zv$. Physically, these relationships illustrate how variations in flow gradients influence the molecular alignment and orientation distribution, particularly across layers in the $z$- direction. This setup forms a feedback loop where the anisotropic stress generated by the Q-tensor not only shapes the flow but is itself modified by the flow gradients. This dynamic provides a model that captures realistic interactions between liquid crystals and surrounding fluid within boundary layer-type regions.

The term $\omega$ in equation \eqref{extra1} captures the influence of fluid vorticity on the orientation tensor of liquid crystals. Specifically, $\omega$ represents a component of the fluid’s vorticity, characterizing the local rotation or curl of the velocity field within the Prandtl layer. In this context, vorticity signifies the extent of rotational motion, often associated with shear forces. Including this vorticity term in the Q-tensor equation emphasizes how the rotational aspects of the fluid flow impact the molecular alignment of the liquid crystals, as captured by the Q-tensor. This influence is particularly pronounced in anisotropic, stratified systems where vorticity-induced reorientations can be significant and directionally dependent.

Physically, this setup suggests that molecular alignment within the liquid crystal system is responsive not only to direct shear or compression forces but also to the rotational dynamics of the surrounding fluid. Such responsiveness is vital for applications that rely on controlled alignment in flowing environments, such as in display technologies or materials engineering, where anisotropic properties and flow-driven alignment are essential.

In the co-rotational framework, the interactions between Q-tensor components and flow gradients can be seen as reflecting the system's tendency to either maintain or resist molecular alignment, depending on local shear and strain rates in the fluid. Therefore, the hydrostatic Beris-Edwards system, when coupled with the Prandtl system, offers a powerful approach for exploring complex interactions between flow and liquid crystal orientations. This model has applications in fields requiring precise control over liquid crystal suspensions or stratified flows, including display technologies, materials science, and biofluid dynamics.

\bigskip
\par\noindent{\bf\large Notations} Before stating our results, we first give some notations and functional spaces that will be used through the paper.
\begin{itemize}
\item[$\bullet$]
For the integral $\int_{\Omega_0\times [0,2\pi\varepsilon]} f dxdydz$ for anisotropic case or $\int_{\Omega_0\times [0,2\pi]} f dxdydz$ for hydrostatic case, we write $\int f$ for simplicity.
\item[$\bullet$]
We write $\partial_1, \partial_2, \partial_3$ for $\partial_x, \partial_y, \partial_z$, $\partial_\va:= (\va\partial_x, \va\partial_y, \partial_z)$ and $\Delta_h$ for $\partial_x^2+\partial_y^2$.
\item[$\bullet$]
We write $a\lesssim b$ to mean that there is a uniform constant $C$, such that $a\le Cb$. Here $C$ may be different on different
lines (but does not depend on $\va$).
\item[$\bullet$]
For any $\alpha, m\ge 0$, we denote 
\[(L^2_\alpha, H^m_\alpha, W^{m,\infty}_\alpha)(\mathbb R):=\{f=f(z), z\in\mathbb R; (||f||_{L^2_\alpha}, ||f||_{H^m_\alpha},  ||f||_{W^{m,\infty}_\alpha}):= ||e^{\alpha z}||_{L^2}, ||e^{\alpha z}||_{H^m}, ||e^{\alpha z}||_{W^{m,\infty}}<\infty\}\]
and for any $\beta>0$, define the functional space
\[
E_{\alpha,\beta}=\left\{f=f(x,y,z)=\sum_{k_1,k_2\in\mathbb Z}e^{i(k_1x+k_2y)}f_{k_1, k_2}(z), ||f_{k_1, k_2}||_{L^2_\alpha}\le C_{\alpha,\beta}e^{-\beta\sqrt{k_1^2+k_2^2}}\right\}
\]
\end{itemize}

\subsection{Main results}
\subsubsection{Anisotropic system}
First, for the scaled anistropic system \eqref{eq3u} and \eqref{Qij1}, we have the following existence theorem.
\begin{thm}
\label{convergenceQ0}

Let $\textbf {u}^\va, Q^\va$ be the solution of \eqref{eq3u}, \eqref{Qij1} with initial data $(\textbf {u}^\va_0, Q^\va_0)$. For any $\nu_1,\nu_2, a,c>0, b\in\mathbb R$, we have\\
(1)\underline{\textit {Exponential growth:}} If $||\textbf u^\va_0(\cdot)||_{L^2(\mathcal S^\va)}+||Q^\va_0(\cdot)||_{L^2(\mathcal S^\va)}+||\partial_\va Q^\va_0(\cdot)||_{L^2(\mathcal S^\va)}\lesssim \va$, then there exists a constant $\lambda_1>0$, such that for any $t\ge 0$,
\begin{equation*}
\begin{aligned}
&||\textbf u^\va(t,\cdot)||_{L^2(\mathcal S^\va)}+||Q^\va(t,\cdot)||_{L^2(\mathcal S^\va)}+||\partial_\va Q^\va(t,\cdot)||_{L^2(\mathcal S^\va)}
\le( ||\textbf u^\va_0(\cdot)||_{L^2(\mathcal S^\va)}+||Q^\va_0(\cdot)||_{L^2(\mathcal S^\va)}+||\partial_\va Q^\va_0(\cdot)||_{L^2(\mathcal S^\va)})e^{\lambda_1t}.
\end{aligned}\end{equation*}
Morever, for $a$ large enough, we have\\
(2)\underline{\textit {Polynomial growth:}} If $||\textbf u^\va_0(\cdot)||_{L^2(\mathcal S^\va)}+||Q^\va_0(\cdot)||_{L^2(\mathcal S^\va)}+||\partial_\va Q^\va_0(\cdot)||_{L^2(\mathcal S^\va)}\lesssim \va^\alpha$, where $\alpha\in (\frac 12, 1)$, then for any $t\ge 0$,
\[
||\textbf u^\va(t,\cdot)||_{L^2(\mathcal S^\va)}+||Q^\va(t,\cdot)||_{L^2(\mathcal S^\va)}+||\partial_\va Q^\va(t,\cdot)||_{L^2(\mathcal S^\va)}
\lesssim ( ||\textbf u^\va_0(\cdot)||_{L^2(\mathcal S^\va)}+||Q^\va_0(\cdot)||_{L^2(\mathcal S^\va)}+||\partial_\va Q^\va_0(\cdot)||_{L^2(\mathcal S^\va)})(1+t)^\frac{\alpha}{1-\alpha}.
\]
(3)\underline{\textit {Exponential decay:}} If $||\textbf u^\va(t,\cdot)||_{L^2(\mathcal S^\va)}+||Q^\va(t,\cdot)||_{L^2(\mathcal S^\va)}+||\partial_\va Q^\va(t,\cdot)||_{L^2(\mathcal S^\va)}\to 0$ as $t\to\infty$, then there exist constants $\lambda, C(\va)>0$, such that for any $t\ge 0$,
\begin{equation*}
\begin{aligned}
&||\textbf u^\va(t,\cdot)||_{L^2(\mathcal S^\va)}+||Q^\va(t,\cdot)||_{L^2(\mathcal S^\va)}+||\partial_\va Q^\va(t,\cdot)||_{L^2(\mathcal S^\va)}
\le C_\va e^{-\lambda_2t}.
\end{aligned}\end{equation*}
\end{thm}
\subsubsection{Hydrostatic system}
For the limit system \eqref{extra1} involving $Q=(Q_{11}(t,\cdot), Q_{12}(t,\cdot), Q_{22}(t,\cdot) )$ and under some assumptions of $(u,v,w)$, we present the following theorem about $H^1$ convergence.
\begin{thm}
\label{convergenceQ1}
Suppose the initial data $(u_{in}, v_{in}, w_{in})$ of \eqref{Prandtl3} satisfy $\partial_xv_{in}-\partial_yu_{in}\in L^{\infty}(\Omega_0\times [0,2\pi])$, and $\theta$ satisfies the following conditions:\\
(1)There exists a small enough constant $\varepsilon'>0$, such that for any $t\ge 0, z\in[0,1]$,
$||\theta|-\sqrt 2/2|\ge \varepsilon', \quad ||\theta|-{\sqrt 2}|\ge \varepsilon'$;\\
(2)$\theta, \partial_z\theta, \partial_t\theta$ are uniformly bounded for any $t\ge 0$ and $z\in [0,1]$.\\
Then for any $\nu_2>0, \varepsilon_0\in (0,\nu_2), b\in\mathbb R)$, there exist constants $A_0, C_0>0$, such that if $a>A_0, c>C_0$, then there exist constants $C,\lambda>0$ such that for any $(Q_{11}, Q_{22}, Q_{12})$ as the solution of \eqref{extra1} and any $t\ge 0$,  we have
\[
||(Q(t,\cdot)||^2_{L^2(\mathcal S)}+\int_0^t{e^{\lambda(s-t)}\left((\nu_2-\va_0) ||\partial_zQ(s,\cdot)||^2_{L^2(\mathcal S)}+||Q^2(s,\cdot)||^2_{L^2(\mathcal S)}\right) ds}\le Ce^{-\lambda t}.
\]
moreover, if we suppose that $\nabla\textbf u,\nabla\omega, \nabla\theta$ are unformly bounded in $\Omega$, then  then there exist constants $C',\lambda'>0$, such that
\[
||\nabla Q(t,\cdot)||^2_{L^2(\mathcal S)}+\int_0^t{e^{\lambda'(s-t)}\left( (\nu_2-\varepsilon_0)||\partial_z\nabla Q(s,\cdot)||^2_{L^2(\mathcal S)}+||Q(s,\cdot)\nabla Q(s,\cdot)||^2_{L^2(\mathcal S)}\right)ds}\le Ce^{-\lambda't}.
\]
\end{thm}
A few remarks are in order:\\
(1)A typical example is $(u,v,w)$ satisfying these assumptions is the Kolmogorov flow $(ae^{-t}\sin z, be^{-t}\sin z, 0)$, where $\left|\frac ab\right|\ne \frac{\sqrt 2}2, \sqrt 2$.\\
(2)Notably, the limit system of  $Q$ does not contain $Q_{13}, Q_{23}$. Although $Q_{13}, Q_{23}$ can be expressed in terms of 
 $Q_{11}, Q_{12}, Q_{22}$, their expressions include singular terms like $\frac{\partial_t\theta}{\partial_zv}$. However, when focusing solely on $Q_{11}, Q_{12}, Q_{22}$, these singularities are avoided.  Therefore, this theorem does not address the convergence of  $Q_{13}, Q_{23}$. Nonetheless, understanding $Q_{13}$ and $Q_{23}$ is still crucial for establishing the convergence of the anisotropic system toward the hydrostatic system, which will be addressed in Lemma \ref{lemtough1}.

\subsubsection{Limit from anisotropic system to hydrostatic system}
Moreover, we explore the convergence of an anisotropic Beris-Edwards system towards a hydrostatic limit in Sobolev space. Extending the result of 2D case in \cite{L1} to 3D presents unique challenges. Let $\f u^{\va}:=u^{\va}-u, \f v^{\va}:=v^{\va}-v, \f w^{\va}:=w^{\va}-w, \f p^{\va}:=p^{\va}-p$. Then $(\f u^\va, \f v^\va, \f w^\va)$ satisfy the following equations: 
\begin{equation}\label{diff1}
\begin{cases}
\partial_t\f u^\va-\varepsilon^2\nu_1\Delta_h \f u^\va-\nu\partial_z^2 \f u^\va+\partial_x\f p^{\va}=R^\va_1\\
\partial_t\f v^\va-\varepsilon^2\nu_1\Delta_h \f v^\va-\nu\partial_z^2 \f v^\va+\partial_y\f p^{\va}=R^\va_2\\
\va\partial_t\f w^\va-\varepsilon^3\nu_1\Delta_h \f u^\va-\va\nu\partial_z^2 \f u^\va+\partial_x\f p^{\va}=R^\va_3\\
\partial_x\f u^\va+\partial_y\f v^\va+\partial_z\f w^\va=0\\
(\f u^\va, \f v^\va, \f w^\va, \partial_z \f u^\va,  \partial_z \f v^\va,  \partial_z \f w^\va)|_{z=0}=(\f u^\va, \f v^\va, \f w^\va, \partial_z \f u^\va,  \partial_z \f v^\va,  \partial_z \f w^\va)|_{z=1}\,\,\text{or}\,\, (\f u^\va, \f v^\va, \f w^\va)|_{z=0,1}=0
\end{cases}
\end{equation}
where
\begin{equation}\label{diff2}
\begin{cases}
R^\va_1:=\va^2\nu_1 \Delta_h u-(u^\va \partial_x u^\va-u\partial_xu)-(v^\va \partial_y u^\va-v\partial_yu)-(w^\va \partial_z u^\va-w\partial_zu)-\varepsilon^4(\partial_xR_{11}+\partial_yR_{12}+\partial_zR_{13})\\
R^\va_2:=\va^2\nu_1 \Delta_h v-(u^\va \partial_x v^\va-u\partial_xv)-(v^\va \partial_y v^\va-v\partial_yv)-(w^\va \partial_z v^\va-w\partial_zv)-\varepsilon^4(\partial_xR_{21}+\partial_yR_{22}+\partial_zR_{23})\\
R^\va_3:=-\va(\partial_t w-\va^2\nu_1\Delta_hw-\nu_1\partial_z^2w+u^\va\partial_xw^\va+v^\va\partial_yw^\va+w^\va\partial_zw^\va)-\varepsilon^4(\partial_xR_{31}+\partial_yR_{32}+\partial_zR_{33})
\end{cases}
\end{equation}
similarly, let $\f Q_{ij}^\varepsilon:=Q_{ij}^\varepsilon-Q_{ij}, 1\le i,j\le 3$. Then $\f Q_{ij}$ satisfy
\begin{equation}\label{diff3}
\begin{cases}
\partial_t\f Q_{ij}^\va-\varepsilon^2\nu_2\Delta_h \f Q_{ij}^\va-\nu_2\partial_z^2 \f Q_{ij}^\va=\f R_{ij} \,\,\text{for}\,\, i,j\in\{1,2\},\\ \varepsilon\partial_t\f Q_{i3}^\va-\varepsilon^3\nu_2\Delta_h \f Q_{i3}^\va-\varepsilon\nu_2\partial_z^2 \f Q_{i3}^\va=\f R_{i3} \,\,\text{for}\,\, i\in\{1,2\} 
\end{cases} \end{equation}
where
\begin{equation}
\label{diff4}
\begin{cases}
\f R_{11}=\varepsilon^2\nu_2\Delta_hQ_{11}-(u^\varepsilon\partial_xQ_{11}^\va-u\partial_xQ_{11})-(v^\varepsilon\partial_yQ_{11}^\va-v\partial_yQ_{11})-(w^\varepsilon\partial_zQ_{11}^\va-w\partial_zQ_{11})\\
+\omega_0^\va Q_{12}^\va-\omega Q_{12}+\omega_1^\va Q_{13}^\va+\partial_zuQ_{13}
-a\f Q_{11}^\va\\+\frac{b}3((Q_{11}^\va)^2-Q_{11}^2+(Q_{12}^\va)^2-Q_{12}^2-2(Q_{22}^\va)^2+2Q_{22}^2-2Q_{11}^\va Q_{22}^\va+2Q_{11}Q_{22}+\va^2(Q_{13}^\va)^2-2\va^2(Q_{23}^\va)^2)\\
-2c\left(Q_{11}^\va((Q_{11}^\va)^2+(Q_{12}^\va)^2+(Q_{22}^\va)^2+Q_{11}^\va Q_{22}^\va+\varepsilon^2 (Q_{13}^\va)^2+\varepsilon^2 (Q_{23}^\va)^2)-Q_{11}(Q_{11}^2+Q_{12}^2+Q_{22}^2+Q_{11}Q_{22})\right)\\
\f R_{12}=\varepsilon^2\nu_2\Delta_hQ_{12}-(u^\varepsilon\partial_xQ_{12}^\va-u\partial_xQ_{12})-(v^\varepsilon\partial_yQ_{12}^\va-v\partial_yQ_{12})-(w^\varepsilon\partial_zQ_{12}^\va-w\partial_zQ_{12})\\
-\frac {\omega_0^\va} 2(Q_{11}^\va-Q_{22}^\va)+\frac {\omega} 2(Q_{11}-Q_{22})+\omega_1^\va Q_{23}^\va+\partial_zuQ_{23}+\omega^\va_2Q_{13}^\va+\partial_zvQ_{13}\\
-a\f Q_{12}^\va+b(Q_{11}^\va Q_{12}^\va-Q_{11}Q_{12}+Q_{12}^\va Q_{22}^\va-Q_{12}Q_{22}+\varepsilon^2 Q_{13}^\va Q_{23}^\va)\\
-2c\left(Q_{12}^\va((Q_{11}^\va)^2+(Q_{12}^\va)^2+(Q_{22}^\va)^2+Q_{11}^\va Q_{22}^\va+\varepsilon^2 (Q_{13}^\va)^2+\varepsilon^2 (Q_{23}^\va)^2)-Q_{12}(Q_{11}^2+Q_{12}^2+Q_{22}^2+Q_{11}Q_{22})\right)\\
\f R_{22}=\varepsilon^2\nu_2\Delta_hQ_{22}-(u^\varepsilon\partial_xQ_{22}^\va-u\partial_xQ_{22})-(v^\varepsilon\partial_yQ_{22}^\va-v\partial_yQ_{22})-(w^\varepsilon\partial_zQ_{22}^\va-w\partial_zQ_{22})\\
-\omega_0^\va Q_{12}^\va+\omega Q_{12}+\omega_2^\va Q_{23}^\va+\partial_zvQ_{23}
-a\f Q_{22}^\va\\+\frac{b}3((Q_{22}^\va)^2-Q_{22}^2+(Q_{12}^\va)^2-Q_{12}^2-2(Q_{11}^\va)^2+2Q_{11}^2-2Q_{11}^\va Q_{22}^\va+2Q_{11}Q_{22}+\va^2(Q_{23}^\va)^2-2\va^2(Q_{13}^\va)^2)\\
-2c\left(Q_{22}^\va((Q_{11}^\va)^2+(Q_{12}^\va)^2+(Q_{22}^\va)^2+Q_{11}^\va Q_{22}^\va+\varepsilon^2 (Q_{13}^\va)^2+\varepsilon^2 (Q_{23}^\va)^2)-Q_{22}(Q_{11}^2+Q_{12}^2+Q_{22}^2+Q_{11}Q_{22})\right)\\
\f R_{13}=-\va(\partial_t Q_{13}-\varepsilon^2\nu_2\Delta_hQ_{13}-\nu_2\partial_z^2Q_{13}+u^{\varepsilon}\partial_xQ^{\varepsilon}_{13}+v^{\varepsilon}\partial_yQ^{\varepsilon}_{13}+w^{\varepsilon}\partial_zQ^{\varepsilon}_{13})\\
+\frac \va 2\omega^\va_0Q^\va_{23}-\frac{\varepsilon}2({(2Q_{11}^\va+Q_{22}^\va)\partial_xw^\va}+Q_{12}^\va\partial_yw^\va)+\frac 1{2\varepsilon}((2Q_{11}^\va+Q_{22}^\va)\partial_zu^\va-(2Q_{11}+Q_{22})\partial_zu+Q^\va_{12}\partial_zv^\va-Q_{12}\partial_zv)\\
+\varepsilon^3 \nu_2\partial_x^2Q^{\varepsilon}_{13}+\varepsilon^3 \nu_2\partial_y^2Q^{\varepsilon}_{13}+\varepsilon \nu_2\partial_z^2Q^{\varepsilon}_{13}-\varepsilon\left(aQ^{\varepsilon}_{13}-b(-Q^\varepsilon_{22}Q^\varepsilon_{13}+Q^\varepsilon_{12}Q^\varepsilon_{23})+cQ_{13}^{\varepsilon}\mbox{tr}(Q^2)\right)\\
\f R_{23}=-\va(\partial_t Q_{23}-\varepsilon^2\nu_2\Delta_hQ_{23}-\nu_2\partial_z^2Q_{23}+u^{\varepsilon}\partial_xQ^{\varepsilon}_{23}+v^{\varepsilon}\partial_yQ^{\varepsilon}_{23}+w^{\varepsilon}\partial_zQ^{\varepsilon}_{23})\\
-\frac \va 2\omega^\va_0Q^\va_{13}-\frac{\varepsilon}2{(Q_{12}^\va\partial_xw^\va}+(Q_{11}^\va+2Q_{22}^\va)\partial_yw^\va)+\frac 1{2\varepsilon}(Q_{12}^\va\partial_zu^\va-Q_{12}\partial_zu+(Q^\va_{11}+2Q^\va_{22})\partial_zv^\va-(Q_{11}+2Q_{22})\partial_zv)\\
+\varepsilon^3 \nu_2\partial_x^2Q^{\varepsilon}_{23}+\varepsilon^3 \nu_2\partial_y^2Q^{\varepsilon}_{23}+\varepsilon \nu_2\partial_z^2Q^{\varepsilon}_{23}-\varepsilon\left(aQ^{\varepsilon}_{23}-b(Q^\varepsilon_{12}Q^\varepsilon_{13}-Q^\varepsilon_{11}Q^\varepsilon_{23})+cQ_{23}^{\varepsilon}\mbox{tr}(Q^2)\right)
\end{cases}
\end{equation}
and $(\f Q_{ij}, \partial_z\f Q_{ij})|_{z=0}= (\f Q_{ij}, \partial_z\f Q_{ij})|_{z=2\pi}\,\,\text{or}\,\, (\f Q_{ij})|_{z=0,2\pi}=0$.\\

Finally, we prove the following theorem concerning the from the scaled anisotropic system   towards the hydrostatic Prandtl system \eqref{Prandtl3} and hydrostatic Q-tensor system \eqref{extra1}.
\begin{thm}
\label{thmconvergence1}
Let $(u,v,w)$ be the solution of \eqref{Prandtl3} and $(Q_{11}, Q_{12}, Q_{22})$ are the solution of \eqref{extra1}, and $(Q_{13}, Q_{23})$ are given by \eqref{toughb5}, \eqref{toughb6}. We suppose that $a>>b, a>>1$.  
Moreover, we give the following assumptions: \\
(1)
\begin{equation}
\label{convergencecondition1}
||(\partial_zu, \partial_zv, \va\partial_zw)||_{L^\infty}, \quad||(Q_{11}, Q_{12}, Q_{22}, \va Q_{13}, \va Q_{23})||_{L^\infty},\quad ||\partial_\va(Q_{11}, Q_{12}, Q_{22}, \va Q_{13}, \va Q_{23})||_{L^\infty}\lesssim\va.
\end{equation}
(2) $\int_0^\infty F^\va(s)ds\lesssim \va^2$, where $F^\va(t)$ is defined in \eqref{toughc18}.\\
Let $(\f u^\va, \f v^\va, \f w^\va, \f Q_{11}^\va, \f Q_{12}^\va, \f Q_{22}^\va, \f Q_{13}^\va, \f Q_{23}^\va)$ be the solution of the equations \eqref{diff1} to \eqref{diff4}, with the initial data \[(\f u^\va_0, \f v^\va_0, \f w^\va_0, \f Q_{11,0}^\va, \f Q_{12,0}^\va, \f Q_{22,0}^\va, \f Q_{13,0}^\va, \f Q_{23,0}^\va).\] If there exists a constant $C_{int}<<\min(\nu_1,\nu_2)$ that does not depend on $\va$, such that
\begin{equation*}
\begin{aligned}
&\f G_0^\va:=||(\f u^\va_0, \f v^\va_0, \va\f w^\va_0)||_{L^2}+||(\f Q_{11,0}^\va, \f Q_{12,0}^\va, \f Q_{22,0}^\va, \va\f Q_{13,0}^\va, \va\f Q_{23,0}^\va)||_{L^2}\\+&||\va\nabla_h(\f Q_{11,0}^\va, \f Q_{12,0}^\va, \f Q_{22,0}^\va, \va\f Q_{13,0}^\va, \va\f Q_{23,0}^\va)||_{L^2}+||\partial_z(\f Q_{11,0}^\va, \f Q_{12,0}^\va, \f Q_{22,0}^\va, \va\f Q_{13,0}^\va, \va\f Q_{23,0}^\va)||_{L^2}\le C_{int}\cdot\va
\end{aligned}
\end{equation*}
then there exists a constant $\f C_l$ that does not depend on $\va$, such that for any $t\ge 0$,
\begin{equation}
\begin{aligned}
&||(\f u^\va, \f v^\va, \va \f w^\va) ||_{L^2}+||(\f Q_{11}^\va, \f Q_{12}^\va, \f Q_{21}^\va, \f Q_{22}^\va, \f Q_{33}^\va ,\va\f Q_{13}^\va, \va \f Q_{23}^\va ,\va\f Q_{31}^\va, \va \f Q_{32}^\va) ||_{L^2}\\
+&||\partial_\va(\f Q_{11}^\va, \f Q_{12}^\va, \f Q_{21}^\va, \f Q_{22}^\va, \f Q_{33}^\va ,\va\f Q_{13}^\va, \va \f Q_{23}^\va ,\va\f Q_{31}^\va, \va \f Q_{32}^\va) ||_{L^2}\lesssim \f G^\va(0)+ \left(\int_0^t F^\va(s)ds\right)^\frac 12.
\end{aligned}
\end{equation}
\end{thm}

\subsection{ Strategies and ideas of the proofs} In this subsection, we provide a brief overview of relevant literature and outline the key strategies employed in proving our results.

The global well-posedness of \eqref{eqbe1} has been established in works such as \cite{PZ1, PZ2}. In our anisotropic system, under suitable scaling, we observe that for sufficiently large $a$, the functional $E(t):=||{\textbf u}^\va||_{L^2}^2+||\partial_\va Q^\va||_{L^2}^2+||\partial_\va Q^\va||_{L^2}^2$ satisfies the differential inequality $\frac{d}{dt}E(t)\lesssim \frac 1{\va^2}E^2(t)$. Consequently,if there exists $p\in (1,2)$, such that $E(t)\lesssim \va^p$ for any $t\ge 0$, then there exists $q\in (0,1)$, such that $\frac d{dt}E(t)\lesssim E(t)^q$, indicating that $E(t)$ exhibits polynomial rather than exponential growth.

In the hydrostatic limit case, due to the chosen scaling, the time derivative terms of $Q_{13}$ and $Q_{23}$ vanish. Moreover, the vorticity of the velocity field reduces to $(\partial_zu, \partial_zv)$, with only the horizontal component $\partial_xv-\partial_yu$ remaining. From $\eqref{extra1}_4, \eqref{extra1}_5$, it becomes clear that $Q_{11}, Q_{12}, Q_{22}$ are directly related to $\partial_zu, \partial_zv$. This allows us to express $Q_{12}$ and $Q_{22}$ in terms of $Q_{11}$ and $\frac{\partial_zu}{\partial_zv}$, i.e. $\theta(t,x,y,z)$.  Additionally,, $Q_{13}$ and $Q_{23}$ can be formulated as functions of $Q_{11}, Q_{12}, Q_{22}$, allowing us to focus primarily on the system forf $(Q_{11}, Q_{12}, Q_{22})$.

A key distinction from the anisotropic system is that, due to the interrelations among $Q_{11}, Q_{12}, Q_{22}$, the standard $H^1$ norm cannot guarantee positive diffusion terms as in the anisotropic case.  Therefore, we introduce distinct weights for $Q_{11}, Q_{12}, Q_{22}$ and perform the analysis within a weighted Sobolev space. These weights are carefully selected to ensure that the transport term integral with the velocity field 
$(u,v,w)$ remains zero and to avoid high-order derivatives in the horizontal directions.

Physically, this convergence result captures the evolution of the Q-tensor in the hydrostatic regime, where stratification leads to significant variation along the vertical axis. In 3D, this setup allows for nontrivial solutions for the Q-tensor in the limit, in contrast to the 2D case \cite{L1}, where the Q-tensor collapses to zero, implying uniform alignment of the liquid crystals. In 3D, however, the compatibility conditions permit structured, nontrivial Q-tensor solutions in the limit, revealing more complex alignment patterns within the anisotropic liquid crystal phase. This contrast illustrates how an additional spatial dimension enables more intricate molecular alignment configurations, even under strong hydrostatic constraints.

Mathematically, proving convergence in 3D Sobolev space requires a more delicate approach. Unlike in 2D \cite{L1}, where the hydrostatic Prandtl system reduces to a shear flow, the 3D hydrostatic Prandtl system can support more general flow structures. This added flexibility necessitates handling additional complexities in the Q-tensor dynamics, particularly concerning vorticity and rotational terms. Moreover, the persistence of nontrivial Q-tensor solutions in the 3D limit demands precise control to rigorously establish convergence. We anticipate similar convergence results for other liquid crystal models, such as the Oldroyd-B model \cite{DP} or Q-tensor systems coupled with the compressible Navier-Stokes equations \cite{GNT}.

To prove the convergence of the anisotropic system to the hydrostatic system, we follow a similar approach to the anisotropic case but address two main differences:\\
1.The presence of terms like $\partial_\va^2\f Q_{ij}^\va\partial_\va \f u^\va Q_{ij}$ necessitates uniform boundedness of $Q_{ij}$.\\
2.The energy functional $\tilde E(t)$ for $\f u^\va, \f v^\va, \f z^\va, \f Q_{ij}^\va$ satisfies $\frac d{dt}\tilde E(t)\lesssim \frac 1{\va^2}\tilde E(t)^2+g(t)$, where $g(t)$ depends on $u,v,w, Q_{11}, Q_{12}, Q_{22}$. To establish the convergence, we assume that $\tilde E(t), g(t)\lesssim \va^2$ uniformly.

Looking forward, extending these convergence results to a general Besov space presents additional mathematical challenges, especially in capturing behavior with lower regularity. Addressing these challenges will allow a full characterization of the convergence properties of the anisotropic Beris-Edwards system in function spaces beyond Sobolev, which we leave for future work.
\bigskip
\par\noindent{\bf\large Structure of the paper} In Section \ref{section2}, we establish the global well-posedness of the anisotropic system. Section \ref{section3} is devoted to proving the global well-posedness and asymptotic behavior of the hydrostatic system. In Section \ref{section4}, we justify the convergence of the anisotropic system toward the hydrostatic system.
\section{Global well-posedness for the anisotropic system}
\label{section2}
In this section, we prove Theorem \ref{convergenceQ0} about the global well-posedness of the anisotropic system.
\begin{proof}We define the matrix $\f S^\va_1, \f S^\va_2, \f S^\va_3$ of the rotation term $Q^\va_{ij}$ for convenience, which is
\begin{equation}
\label{matrix1}
\f S^\va_1:=\omega^\va_0
\begin{pmatrix}
-Q_{12}^\va&
\frac 12(Q^\va_{11}-Q^\va_{22})&
-\frac {\va}2 Q_{23}^\va\\
\frac 12(Q^\va_{11}-Q^\va_{22})&
Q_{12}^\va&
\frac{\varepsilon}2Q^\varepsilon_{13}\\
-\frac {\va}2Q_{23}^\va&
\frac{\varepsilon}2Q^\varepsilon_{13}&0
\end{pmatrix}
\end{equation}
\begin{equation}
\label{matrix2}
\f S^\va_2:=\omega^\va_1
\begin{pmatrix}
-Q^\va_{13}&
-\frac 12Q_{23}^\va&
\frac 1{2\va}( Q^\va_{11}-Q^\va_{33})\\
-\frac 12Q_{23}^\va&
0&
\frac{1}{2\varepsilon}Q^\varepsilon_{12}\\
\frac 1{2\va}( Q^\va_{11}-Q^\va_{33})&
\frac{1}{2\varepsilon}Q^\varepsilon_{12}&
Q^\va_{13}
\end{pmatrix}
\end{equation}
\begin{equation}
\label{matrix3}
\f S^\va_3:=\omega^\va_2
\begin{pmatrix}
0&
-\frac 12Q_{13}^\va&
\frac 1{2\va}Q^\va_{12}\\
-\frac 12Q_{13}^\va&
-Q_{23}^\va&
\frac{1}{2\varepsilon}(Q_{22}^\varepsilon-Q_{33}^\varepsilon)\\
\frac 1{2\va}Q_{12}&
\frac{1}{2\varepsilon}(Q_{22}^\varepsilon-Q_{33}^\varepsilon)&
Q_{23}^\va
\end{pmatrix}
\end{equation}
set
\begin{equation}
\begin{aligned}
\f F^\va(t):=&\frac 12||(u^\va, v^\va, \f w^\va) ||_{L^2}^2+\frac 12||(Q_{11}^\va, Q_{12}^\va, Q_{21}^\va, Q_{22}^\va, Q_{33}^\va ,Q_{13}^\va, Q_{23}^\va ,\va Q_{31}^\va, \va  Q_{32}^\va) ||_{L^2}^2\\
+&\frac 12||\partial_\va(Q_{11}^\va, Q_{12}^\va, Q_{21}^\va, Q_{22}^\va,  Q_{33}^\va ,\va Q_{13}^\va, \va  Q_{23}^\va ,\va Q_{31}^\va, \va  Q_{32}^\va) ||_{L^2}^2
\end{aligned}
\end{equation}
because of the divergence free condition of $\textbf u^\va$ and boundary condition of $\textbf u^\va, Q^\va_{ij}$, we integrate by parts to get
\begin{equation}
\label{anisotropic1}
\begin{aligned}
&\frac d{dt}\f F^\va(t)+\nu_1||\partial_\va(u^\va, v^\va, \va w^\va) ||_{L^2}^2+ a||(\f Q_{11}^\va, \f Q_{12}^\va, \f Q_{21}^\va, \f Q_{22}^\va, \f Q_{33}^\va ,\va\f Q_{13}^\va, \va \f Q_{23}^\va ,\va\f Q_{31}^\va, \va \f Q_{32}^\va) ||_{L^2}^2\\
&+(a+\nu_2)||\partial_\va(\f Q_{11}^\va, \f Q_{12}^\va, \f Q_{21}^\va, \f Q_{22}^\va, \f Q_{33}^\va ,\va\f Q_{13}^\va, \va \f Q_{23}^\va ,\va\f Q_{31}^\va, \va \f Q_{32}^\va) ||_{L^2}^2\\
&+\nu_2||\partial_\va^2(Q_{11}^\va, Q_{12}^\va, Q_{21}^\va, Q_{22}^\va,  Q_{33}^\va ,\va Q_{13}^\va, \va Q_{23}^\va ,\va Q_{31}^\va, \va Q_{32}^\va) ||_{L^2}^2:=\sum_{i=1}^{14} T_{i}
\end{aligned}
\end{equation}
where
\begin{equation}
\label{anisotropic2}
\begin{aligned}
T_1:=&\va^4(M_{11},\partial_x u^\va)_{L^2}+\va^4(M_{21},\partial_y u^\va)_{L^2}+\va^4(M_{31},\partial_z u^\va)_{L^2}+\va^4(M_{12},\partial_x v^\va)_{L^2}\\+&\va^4(M_{22},\partial_y v^\va)_{L^2}+\va^4(M_{32},\partial_z v^\va)_{L^2}+\va^4(M_{13},\partial_x w^\va)+\va^4(M_{32},\partial_y w^\va)+\va^4(M_{33},\partial_z w^\va)
\end{aligned}
\end{equation}
\end{proof}
\begin{equation}
\label{anisotropic3}
\begin{aligned}
T_{2}:=&-\sum_{1\le i,j\le 2}(\partial_\va u^\va\partial_x Q_{ij}^\va +\partial_\va v^\va\partial_y Q_{ij}^\va+\partial_\va w^\va\partial_z Q_{ij}^\va, \partial_\va Q_{ij}^\va)_{L^2}\\
&-2\sum_{j=1}^2\va^2(\partial_\va u^\va\partial_x Q_{i3}^\va + \partial_\va v^\va\partial_y Q_{i3}^\va+\partial_\va w^\va\partial_z Q_{i3}^\va, \partial_\va Q_{i3}^\va)_{L^2}
\end{aligned}
\end{equation}
\begin{equation}
\label{anisotropic4}
\begin{aligned}
T_3:&=\sum_{k=1}^3\sum_{(i,j)\ne (1,3), (2,3)}(\f S^\va_k, Q^\va_{ij})_{L^2}+2\va \sum_{k=1}^3\sum_{(i,j)= (1,3)\,\,\text{or}\,\, (2,3)}(\f S^\va_k, Q^\va_{ij})_{L^2}
\end{aligned}
\end{equation}
\begin{equation}
\label{anisotropic5}
\begin{aligned}
T_{4}:&=\sum_{k=1}^3\sum_{(i,j)\ne (1,3), (2,3)}(\partial_\va\f S^\va_k, \partial_\va Q^\va_{ij})_{L^2}+2\va\sum_{k=1}^3\sum_{(i,j)=(1,3)\,\,\text{or}\,\, (2,3)}(\partial_\va\f S^\va_k, \partial_\va Q^\va_{ij})_{L^2}
\end{aligned}
\end{equation}
\begin{equation}
\label{anisotropic6}
\begin{aligned}
T_{5}:&=b\int{(Q^\va)^3-\frac 13 Q^\va\text{tr}((Q^\va)^2)}-c\int {(Q^\va)^2\text{tr}((Q^\va)^2)}
\end{aligned}
\end{equation}
\begin{equation}
\label{anisotropic7}
\begin{aligned}
T_{6}:&=b\int{Q^\va(\partial_\va Q^\va)^2-\frac 23 Q^\va\text{tr}(\partial_\va Q^\va\cdot Q^\va)}-c\int {(\partial_\va Q^\va)^2\text{tr}((Q^\va)^2)+2(\partial_\va Q^\va\cdot Q^\va)\cdot\text{tr}(\partial_\va Q^\va\cdot Q^\va)}
\end{aligned}
\end{equation}
from Cauchy-Schwartz inequality, we have
\begin{equation*}
T_1\lesssim \va^2||\partial_\va^2(Q_{11}^\va, Q_{12}^\va, Q_{21}^\va, Q_{22}^\va,  Q_{33}^\va ,\va Q_{13}^\va, \va Q_{23}^\va ,\va Q_{31}^\va, \va Q_{32}^\va) ||_{L^2}\cdot||\partial_\va(u^\va, v^\va, \va w^\va) ||_{L^2}
\end{equation*}
\begin{equation*}
T_2\lesssim \frac 1{\va}||\partial_\va(Q_{11}^\va, Q_{12}^\va, Q_{21}^\va, Q_{22}^\va,  Q_{33}^\va ,\va Q_{13}^\va, \va Q_{23}^\va ,\va Q_{31}^\va, \va Q_{32}^\va) ||^2_{L^2}\cdot||\partial_\va(u^\va, v^\va, \va w^\va) ||_{L^2}
\end{equation*}
from the structure of $\f S^\va_1, \f S^\va_2, \f S^\va_3$, we have $T_3=0$. Moreover, from integration by parts, we obtain that
\begin{equation}
\begin{aligned}
T_4:=&(\omega^\va_0Q^\va_{12},\partial_\va^2(Q_{11}^\va-Q_{22}^\va))_{L^2}-(\omega^\va_0(Q^\va_{11}-Q^\va_{22}),\partial_\va^2Q_{12}^\va)_{L^2}+\va^2(\omega^\va_0Q^\va_{23},\partial_\va^2Q_{13}^\va)_{L^2}-\va^2(\omega^\va_0Q^\va_{13},\partial_\va^2Q_{23}^\va)_{L^2}\\
+&(\omega^\va_1Q^\va_{13},\partial_\va^2(Q_{11}^\va-Q_{33}^\va))_{L^2}-(\omega^\va_1(Q^\va_{11}-Q^\va_{33}),\partial_\va^2Q_{13}^\va)_{L^2}+\va^2(\omega^\va_1Q^\va_{12},\partial_\va^2Q_{23}^\va)_{L^2}-\va^2(\omega^\va_1Q^\va_{23},\partial_\va^2Q_{12}^\va)_{L^2}\\
+&(\omega^\va_2Q^\va_{23},\partial_\va^2(Q_{22}^\va-Q_{33}^\va))_{L^2}-(\omega^\va_2(Q^\va_{22}-Q^\va_{33}),\partial_\va^2Q_{23}^\va)_{L^2}+\va^2(\omega^\va_2Q^\va_{13},\partial_\va^2Q_{12}^\va)_{L^2}-\va^2(\omega^\va_2Q^\va_{12},\partial_\va^2Q_{13}^\va)_{L^2}\\
\lesssim&\frac 1{\va}||(Q_{11}^\va, Q_{12}^\va, Q_{21}^\va, Q_{22}^\va, Q_{33}^\va ,Q_{13}^\va, Q_{23}^\va ,\va Q_{31}^\va, \va  Q_{32}^\va) ||_{L^\infty}\cdot\\
&||\partial_\va(\f u^\va, \f v^\va, \va\f w^\va)||_{L^2}\cdot||\partial_\va^2(Q_{11}^\va, Q_{12}^\va, Q_{21}^\va, Q_{22}^\va,  Q_{33}^\va ,\va Q_{13}^\va, \va Q_{23}^\va ,\va Q_{31}^\va, \va Q_{32}^\va) ||_{L^2}\\
\lesssim&\frac{\nu_2}{100}||\partial_\va^2(Q_{11}^\va, Q_{12}^\va, Q_{21}^\va, Q_{22}^\va,  Q_{33}^\va ,\va Q_{13}^\va, \va Q_{23}^\va ,\va Q_{31}^\va, \va Q_{32}^\va) ||_{L^2}^2\\
+&\frac 1{\va^2}||\partial_\va(\f u^\va, \f v^\va, \va\f w^\va)||_{L^2}^2||(Q_{11}^\va, Q_{12}^\va, Q_{21}^\va, Q_{22}^\va, Q_{33}^\va ,Q_{13}^\va, Q_{23}^\va ,\va Q_{31}^\va, \va  Q_{32}^\va) ||_{L^2}^2\\
+&\frac 1{\va^2}||\partial_\va(\f u^\va, \f v^\va, \va\f w^\va)||_{L^2}^2||\partial_z(Q_{11}^\va, Q_{12}^\va, Q_{21}^\va, Q_{22}^\va, Q_{33}^\va ,Q_{13}^\va, Q_{23}^\va ,\va Q_{31}^\va, \va  Q_{32}^\va) ||_{L^2}^2
\end{aligned}
\end{equation}
for the terms $T_5$ and $T_6$, we have
\begin{equation*}
T_5+T_6+\frac c2\int {(Q^\va)^2\text{tr}((Q^\va)^2)}+\frac c2\int {(\partial_\va Q^\va)^2\text{tr}((Q^\va)^2)}\lesssim ||Q^\va||^2_{L^2}+||\partial_\va Q^\va||_{L^2}^2
\end{equation*}
remind that with the boundary conditions of ${\textbf u}^\va, Q^\va$, we have the Poincar\'e inequality
\[
||(u^\va, v^\va, \va w^\va, Q_{11}^\va,  Q_{12}^\va,  Q_{22}^\va, \va Q_{13}^\va,   \va Q_{23}^\va)||^2_{L^2}\lesssim ||\partial_z(u^\va, v^\va, \va w^\va, Q_{11}^\va,  Q_{12}^\va,  Q_{22}^\va, \va Q_{13}^\va,   \va Q_{23}^\va)||^2_{L^2}
\]
combine the estimates above together, we have for any $\nu_1,\nu_2, a, c>0, b\in \mathbb R$, there exists a constant $\lambda_0> 0$ that does not depend on $\va$, such that $\frac d{dt}\f F^\va(t)\le \lambda_0\f F^\va(t)+\frac 1{\va^2} (\f F^\va(t))^2$. Under the assumption $\f F^\va(t)\lesssim \va$, there exists a constant $\lambda_1>0$, such that $\f F^\va(t)\lesssim e^{\lambda _1t}\f F^\va(0) $.\\
If $a$ is large enough, then there exist a constant $\lambda'$, such that \[\frac{d}{dt}\f F^\va(t)\lesssim -\lambda'\f F^\va(t)+\frac 1{\va^2}\f (\f F^\va(t))^2\le -\frac 1{\va^2}\f (\f F^\va(t))^2.\] Under the assumption that $\f F^\va (t)\lesssim \va^p (1<p<2)$, we obtain that $\frac d{dt}\f F^\va(t)\lesssim (\f F^\va(t))^\frac{2p-2}{p}$. By solving this ODE, we obtain that there exists a constant $C>0$, such that
\[
\f F^\va(t)\le \left(\f F^\va(0)^\frac{2-p}{p}+\frac{C(2-p)}{p}\right)^\frac{p}{2-p}.\]
If $\f F^\va(t)\to 0$ as $t\to\infty$, then exists $T^\va$, such that $\f F^\va(t)< \lambda'\va^2$ for any $t\ge T^\va$. From Gronwall inequality, we have for any $t\ge T^\va $,
\[
\f F^\va (t)\le\left (e^{\lambda(t-T^\va)}\frac{\lambda'\va^2-\f F^\va(T^\va)}{\lambda'\va^2}+\frac{1}{\lambda'\va^2}\right)^{-1}<\frac{\lambda'\va^2e^{\lambda' T^\va}}{\lambda'\va^2-\f F^\va(T^\va)}e^{-\lambda't}.
\]
This finishes the proof of Theorem \ref{convergenceQ0}.
\section{Global well-posedness for the hydrostatic Q-tensor system}
\label{section3}
In this section we study the global well-posedness of the hydrostatic Q-tensor system \eqref{extra1}, which depends the behaviour of the Prandtl system \eqref{Prandtl3}.  
\subsection{More information about Prandtl equation}
For $(x,y)\in [0,2\pi]$, the general 3D Prandtl equation reads
\begin{equation}
\label{ill1}
\begin{cases}
\partial_tu+u\partial_xu+v\partial_yu+w\partial_zu=\partial_z^2u\\
\partial_tv+u\partial_xv+v\partial_yv+w\partial_zv=\partial_z^2v\\
\partial_xu+\partial_yv+\partial_zw=0\\
(u,v,w)|_{z=0}=0,  \lim_{z\to\infty}=(U_0, V_0), U_0,V_0>0\\
\end{cases}
\end{equation}
and our case restricts the vertical variable $z$ in an interval (here we suppose that $\nu_2=1$). A special solution of \eqref{ill1} is the shear flow $(u^s(t,z), v^s(t,z))$, where $u^s, v^s$ satisfy the following heat equation:
\begin{equation}
\label{ill2}
\begin{cases}
\partial_tu^s=\partial_z^2u^s, \quad \partial_tv^s=\partial_z^2v^s\\
(u^s, v^s)|_{z=0}=0, \lim_{z\to \infty}(u^s, v^s)=(U^0, V^0), (u^s, v^s)|_{t=0}=(U^s, V^s)(z)
\end{cases}
\end{equation}
with $(u^s-U_0, v^s-V_0)$ rapidly tending to 0 when $z\to\infty$, and the linearization of \eqref{ill1} around $(u^s, v^s, 0)$ becomes
\begin{equation}
\label{ill3}
\begin{cases}
\partial_tu+(u^s\partial_x+v^s\partial_y)u+wu^s_z=\partial_z^2u\\
\partial_tv+(u^s\partial_x+v^s\partial_y)v+wv^s_z=\partial_z^v\\
\partial_xu+\partial_yv+\partial_zw=0\\
(u,v,w)|_{z=0}=0, \lim_{z\to\infty}(u,v)=0
\end{cases}
\end{equation}
in our case, a classical example of shear flow will be Kolmogorov flow $(\tilde u(t,z), \tilde v(t,z), 0)$, when
\[
\tilde u(t,z)=\sum_{n=1}^\infty a_ne^{-n^2t}\sin (nz), \quad \tilde v(t,z)=\sum_{n=1}^\infty b_ne^{-n^2t}\sin (nz), a_n, b_n\in\mathbb R.
\]
with initial data
\[
\tilde u_0(z)=\sum_{n=1}^\infty a_n\sin (nz), \quad \tilde v_0(z)=\sum_{n=1}^\infty b_n\sin (nz).
\]
however, because the Kolmogorov flow decays exponentially and there does not exist viscosities on horizontal variables, we could not expect that the Kolmogorov flow above are stable (either linear or nonlinear) like Euler equation or Navier-Stokes equation \cite{BL,WZZ2}.  In fact, because $\tilde u_0, \tilde v_0$ have nondegenerate critical points,  there exists \textit{ill-posedness} (see \cite{GD} for example).

Moreover, \cite{LXY1} showed even if we revise the interval to let $\tilde u_0,\tilde v_0$ be monotone, if there exists $z_0>0$, such that $\partial_z(\frac{\partial_z\tilde u_0}{\partial_z\tilde v_0}) (z_0)=0$, there still exists ill-posedness. Because the behaviour of $\frac{\partial_zu}{\partial_zv}$ has direct relations with the behaviour of hydrostatic system \eqref{extra1}, we give more details here. Denote operator $T\in \mathcal (E_{\alpha,\beta}, E_{\alpha, \beta'})$ as $T(t,s)((u_0,v_0)):=(u,v)(t,\cdot)$, where $(u,v)$ is the solution of \eqref{ill3} with $(u,v)|_{t=s}=(u_0,v_0)$. Moreover, for $\alpha,m\ge 0$, we define functional spaces
\[
\mathcal H^m_\alpha:= H^m([0,2\pi]\times[0,2\pi]|_{x,y}; L^2_\alpha(\mathbb R_z)).
\]
Because $E_{\alpha,\beta}$ is dense in $\mathcal H^m_\alpha$, we can extend the operator $T$ from the space $E_{\alpha,\beta}$ to $\mathcal H^m_\alpha$, and define
\[
||T(t,s)||_{\mathcal L(H^{m_1}_\alpha, \mathcal H^{m_2}_\alpha)}:= \sup_{u_0,v_0\in E_{\alpha,\beta}}\frac{||T(t,s)(u_0,v_0)||_{\mathcal H^{m_2}_\alpha}}{||(u_0,v_0)||_{\mathcal H^{m_1}_\alpha}}\in \mathbb R^+\cup \{\infty\}
\]
where the infinity means that $T$ can not be extended to ${\mathcal L(\mathcal H^{m_1}_\alpha, \mathcal H^{m_2}_\alpha)}$. The main result can be stated as follows:
\begin{thm}
(\cite{LXY1}, Theorem 1) Let $(u^s, v^s)(t,z)$ be the solution of \eqref{ill2} satisfying
\[
(u^s-U_0, v^s-V_0)\in C^0(\mathbb R^+; W^{4,\infty}_\alpha(\mathbb R^+)\cap H^4_\alpha(\mathbb R^+))\cap C^1(\mathbb R^+;W^{2,\infty}_\alpha(\mathbb R^+)\cap H^2_\alpha(\mathbb R^+)).
\]
if the initial data $U_s, V_s$ satisfies $\frac{d}{dz}\left(\frac{U'_s}{V'_s}\right)\not\equiv 0$, then
(1)There exists $\sigma>0$, such that for all $\sigma>0$,
\[
\sup_{0\le s\le t\le\delta}||e^{-\sigma(t-s)\sqrt{|\partial_{\mathcal T}|}}T(t,s)||_{\mathcal H^m_\alpha, \mathcal H^{m-\nu}_\alpha}=\infty,\quad \forall m>0,\mu\in [0,\frac 14),
\]
where $\partial_{\mathcal T}$ represents the tangential derivative $\partial_x$ or $\partial_y$;
(2) There exists an initial shear layer $(U_s, V_s)$ to \eqref{ill2} and $\sigma>0$, such that for any $\delta>0$,
\[
\sup_{0\le s\le t\le\delta}||e^{-\sigma(t-s)\sqrt{|\partial_{\mathcal T}|}}T(t,s)||_{\mathcal H^{m_1}_\alpha, \mathcal H^{m_2}_\alpha}=\infty,\quad \forall m_1, m_2>0.
\]
\end{thm}
so a possible choice for the stability of \eqref{extra1} around Kolmogorov flow will be alway supposing $u=p_1(t,x,y)v+p_2(t,x,y)$. From $\eqref{extra1}_1, \eqref{extra1}_2$, we have $\partial_tp_1\cdot v+\partial_tp_2=0$, which means that $\partial_tp_1=\partial_tp_2=0$, i.e. $p_1, p_2$ are time-independant.

Another important problem about Prandtl equation is the stationary solutions. Let us start by simplified 2D model
\begin{equation}
\label{Prandtl2D1}
\begin{cases}
u\partial_xu+v\partial_yu=\partial_y^2u,\quad x,y\in \mathbb R^+,\\
\partial_xu+\partial_yv=0,\\
u|_{x=0}=u_0(y), \quad (u,v)|_{y=0}=0, \quad\lim_{y\to\infty}u(x,y)=1
\end{cases}
\end{equation}
the stationary state of \eqref{Prandtl2D1} is the \textit{Blasius self similar solution} $f$ that satisfies the following ODE:
\begin{equation}
\label{Prandtl2D2}
f'''+\frac 12ff''=0, x\in \mathbb R^+;\quad f(0)=f'(0)=0, \quad f'(+\infty)=1
\end{equation}
then the vector field 
\begin{equation}
\label{Prandtl2D3}
(\bar u,\bar v):=(f'(\eta), \frac 1{2\sqrt {x+1}}(\eta f'(\eta)-f(\eta)))),\quad \eta=\frac{y}{\sqrt {x+1}} 
\end{equation}
is the solution of \eqref{Prandtl2D1}. A natural question is the convergence solutions of Prandtl's equation \eqref{Prandtl2D1} converges to Blasius profile and the rate of convergence. The authors proved in \cite{WZZ3} that $
||u-\bar u||_{L^\infty_y}\lesssim (x+1)^{-(\frac 12)-}$,
and recently it was showed in \cite{JLY} that this result can be improved by $||u-\bar u||_{L^\infty_y}\lesssim (x+1)^{-1}$. However, due to the complexity of 3D Prandtl system, the information self-similar stationary solutions remains open.

Now we comeback to \eqref{Prandtl3} to prove the $L^\infty$ boundness of $\omega$.
\begin{prop}
\label{vorticity1}
Under the assumption of Theorem \eqref{convergenceQ1}, we have $\partial_xv(t,\cdot)-\partial_yu(t,\cdot)\in L^{\infty}(\mathcal S)$ for any $t\ge 0$.
\end{prop}
\begin{proof}
From direct calculations, $\partial_t\omega+u\partial_x\omega+v\partial_y\omega+\partial_xw\partial_zv-\partial_yw\partial_zu= \nu_1\partial_{z}^2\omega$. Then for any $k\in\mathbb N$,
\begin{equation*}
\begin{aligned}
\frac{1}{2k+2}\frac d{dt}\int\omega^{2k+2} =-(2k+1) \nu_1\int |\omega^k\partial_z\omega|^2\le 0
\end{aligned}
\end{equation*}
thus $ ||\omega(t,\cdot)||_{L^{2k+2}(\Omega_0\times [0,2\pi])}\le ||\omega_0||_{L^{2k+2}(\Omega_0\times [0,2\pi])}$ for any $t\ge 0$. And the proposition is proved by letting $k\to\infty$.
\end{proof}
\subsection{Convergence of hydrostatic Q-tensor system}
In this subsection  we prove Theorem \ref{convergenceQ1}.

\begin{proof}
We begin by rewriting schematically the equations $\eqref{extra1}_1, \eqref{extra1}_2, \eqref{extra1}_3$ as the following forms:
\begin{equation}
\label{limq11+}
\partial_t Q_{11}=F_{11}-Q_{13}\partial_zu
\end{equation}
\begin{equation}
\label{limq12+}
\partial_tQ_{12}=F_{12}-Q_{23}\partial_zu-Q_{13}\partial_zv
\end{equation}
\begin{equation}
\label{limq22+}
\begin{aligned}
\partial_tQ_{22}=F_{22}-Q_{23}\partial_zv\\
\end{aligned}
\end{equation} where $F_{11},F_{12},F_{22}$ are independent of terms containing $Q_{13},Q_{23}$ and are also independent of 
$\partial_t Q_{11},\partial_t Q_{12}, \partial_t Q_{22}$.
Take the time derivative of $\eqref{extra1}_4$, we obtain
\begin{equation}
\label{limq13tderiv}
\partial_t (2Q_{11}+Q_{22})\partial_zu+\partial_t Q_{12}\partial_zv+\partial^2_{tz}u(2Q_{11}+Q_{22})+\partial^2_{tz}vQ_{12} =0
\end{equation}
we calculate \eqref{limq11+} $\times 2\partial_zu+$\eqref{limq12+} $\times \partial_zv$+$\eqref{limq22+}\times\partial_zu$ and we obtain:
\begin{align}
\label{scheme1}
\partial_t (2Q_{11}+Q_{22})\partial_zu+\partial_t Q_{12}\partial_zv=2F_{11}\partial_zu-2(\partial_zu)^2Q_{13}+F_{12}\partial_zv-2\partial_zu\partial_zvQ_{23}-(\partial_zv)^2Q_{13}+F_{22}\partial_zu
\end{align}
\eqref{scheme1}-\eqref{limq13tderiv}, we have
\begin{equation}
\label{scheme01}
[2(\partial_zu)^2+(\partial_zv)^2]Q_{13}+\partial_zu\partial_zvQ_{23}=2F_{11}\partial_zu+F_{12}\partial_zv+F_{22}\partial_zu-\partial_t\partial_zu(2Q_{11}+Q_{22})-\partial_t\partial_zvQ_{12}:=M_1
\end{equation}
similarly, we do the time derivative of $\eqref{extra1}_5$ to get that
\begin{equation}\label{limq23deriv1}
    \partial_t Q_{12}\partial_zu+Q_{12}\partial_t\partial_zu +(\partial_t Q_{11}+2\partial_t Q_{22})\partial_zv+(Q_{11}+2Q_{22})\partial_t\partial_zv=0
\end{equation}
we calculate \eqref{limq11+} $\times \partial_zv+$\eqref{limq12+} $\times \partial_zu$+$\eqref{limq22+}\times2\partial_zv$ and we obtain:
\begin{align}
\label{scheme2}
(\partial_t Q_{11}+2\partial_t Q_{22})\partial_zv\partial_t Q_{12}\partial_zu=F_{11}\partial_zv-2Q_{13}\partial_zu\partial_zv+F_{12}\partial_zu-Q_{23}(\partial_zu)^2+2F_{22}\partial_zv-2Q_{23}(\partial_zv)^2
\end{align}
\eqref{scheme2}-\eqref{limq23deriv1}, we have
\begin{equation}
\label{scheme02}
2\partial_zu\partial_zvQ_{13}+[(\partial_zu)^2+2(\partial_zv)^2]Q_{23}=F_{11}\partial_zv+F_{12}\partial_zu+2F_{22}\partial_zv-Q_{12}\partial_t\partial_zu-(Q_{11}+2Q_{22})\partial_t\partial_zv:=M_2
\end{equation}
notice that \eqref{scheme01}, \eqref{scheme02} is the equation set of $Q_{13}, Q_{23}$. Set
$A_{11}:=2(\partial_z u)^2+(\partial_z v)^2, \, A_{12}:=2\partial_z u\partial_z v,\, A_{22}:=(\partial_z u)^2+2(\partial_z v)^2, \mathfrak A:=A_{11}A_{22}-A^2_{12}=2(\partial_z u)^4+2(\partial_z v)^4+(\partial_z u)^2(\partial_z v)^2> 0$. Then $Q_{13} A_{11}+Q_{23}A_{12}=M_1, Q_{13} A_{21}+Q_{23}A_{22}=M_2$.  Thus
\begin{equation}
\label{scheme1}
Q_{13}=\frac 1{\mathfrak A}(A_{22}M_1-A_{12}M_2), \quad Q_{23}=\frac 1{\mathfrak A}(-A_{12}M_1+A_{11}M_2)
\end{equation}
recall \eqref{limq11+}-\eqref{limq22+}, we have
\begin{equation}
\label{scheme2}
\begin{aligned}\partial_tQ_{11}=&F_{11}-\frac{\partial_zu}{\mathfrak A}[((\partial_zu)^2+2(\partial_zv)^2))\cdot ((2F_{11}+F_{22})\partial_zu+F_{12}\partial_zv-(2Q_{11}+Q_{22})\partial_t\partial_zu-Q_{12}\partial_t\partial_zv)]\\
+&\frac{\partial_zu}{\mathfrak A}[2\partial_zu\partial_zv\cdot (F_{12}\partial_zu+(F_{11}+2F_{22})\partial_zv-Q_{12}\partial_t\partial_zu-(Q_{11}+2Q_{22})\partial_t\partial_zv)]\\
\end{aligned}
\end{equation}
\begin{equation}
\label{scheme3}
\begin{aligned}
\partial_tQ_{22}=&F_{22}-\frac{\partial_zv}{\mathfrak A}[(2(\partial_zu)^2+(\partial_zv)^2))\cdot (F_{12}\partial_zu+(F_{11}+2F_{22})\partial_zv-Q_{12}\partial_t\partial_zu-(Q_{11}+2Q_{22})\partial_t\partial_zv)]\\
+&\frac{\partial_zv}{\mathfrak A}[2\partial_zu\partial_zv\cdot ((2F_{11}+F_{22})\partial_zu+F_{12}\partial_zv-(2Q_{11}+Q_{22})\partial_t\partial_zu-Q_{12}\partial_t\partial_zv)]\\
\end{aligned}
\end{equation}
\begin{equation}
\label{scheme3}
\begin{aligned}
&\partial_tQ_{12}=F_{12}-\frac{\partial_zu}{\mathfrak A}[(2(\partial_zu)^2+(\partial_zv)^2))\cdot (F_{12}\partial_zu+(F_{11}+2F_{22})\partial_zv-Q_{12}\partial_t\partial_zu-(Q_{11}+2Q_{22})\partial_t\partial_zv)]\\
&+\frac{\partial_zu}{\mathfrak A}[2\partial_zu\partial_zv\cdot ((2F_{11}+F_{22})\partial_zu+F_{12}\partial_zv-(2Q_{11}+Q_{22})\partial_t\partial_zu-Q_{12}\partial_t\partial_zv)]\\
&-\frac{\partial_zv}{\mathfrak A}[((\partial_zu)^2+2(\partial_zv)^2))\cdot ((2F_{11}+F_{22})\partial_zu+F_{12}\partial_zv-(2Q_{11}+Q_{22})\partial_t\partial_zu-Q_{12}\partial_t\partial_zv)]\\
&+\frac{\partial_zv}{\mathfrak A}[2\partial_zu\partial_zv\cdot (F_{12}\partial_zu+(F_{11}+2F_{22})\partial_zv-Q_{12}\partial_t\partial_zu-(Q_{11}+2Q_{22})\partial_t\partial_zv)]\\
\end{aligned}
\end{equation}
notice that from \eqref{limq13tderiv}, \eqref{limq23deriv1}, we have
\[
\partial_{t}\partial_zu(2Q_{11}+Q_{22})+\partial_{t}\partial_zvQ_{12} =-\partial_t (2Q_{11}+Q_{22})\partial_zu-\partial_t Q_{12}\partial_zv,
\]
\[
 \partial_{t}\partial_zuQ_{12}+\partial_{t}\partial_zv(Q_{11}+2Q_{22}) =-\partial_t Q_{12}\partial_zu-\partial_t( Q_{11}+2Q_{22})\partial_zv
\]
it means that
\begin{equation}
\label{scheme5}
\begin{aligned}
&\partial_tQ_{11}=F_{11}-\frac{\partial_zu}{\mathfrak A}[((\partial_zu)^2+2(\partial_zv)^2))\cdot ((2F_{11}+F_{22})\partial_zu+F_{12}\partial_zv+\partial_t (2Q_{11}+Q_{22})\partial_zu+\partial_t Q_{12}\partial_zv)]\\
&+\frac{\partial_zu}{\mathfrak A}[2\partial_zu\partial_zv\cdot (F_{12}\partial_zu+(F_{11}+2F_{22})\partial_zv+\partial_t Q_{12}\partial_zu+\partial_t( Q_{11}+2Q_{22})\partial_zv)]\\
\end{aligned}
\end{equation}
\begin{equation}
\label{scheme6}
\begin{aligned}
&\partial_tQ_{22}=F_{22}-\frac{\partial_zv}{\mathfrak A}[(2(\partial_zu)^2+(\partial_zv)^2))\cdot (F_{12}\partial_zu+(F_{11}+2F_{22})\partial_zv+\partial_t Q_{12}\partial_zu+\partial_t( Q_{11}+2Q_{22})\partial_zv)]\\
&+\frac{\partial_zv}{\mathfrak A}[2\partial_zu\partial_zv\cdot ((2F_{11}+F_{22})\partial_zu+F_{12}\partial_zv+\partial_t (2Q_{11}+Q_{22})\partial_zu+\partial_t Q_{12}\partial_zv)]\\
\end{aligned}
\end{equation}
\begin{equation}
\label{scheme7}
\begin{aligned}
&\partial_tQ_{12}=F_{12}-\frac{\partial_zu}{\mathfrak A}[(2(\partial_zu)^2+(\partial_zv)^2))\cdot (F_{12}\partial_zu+(F_{11}+2F_{22})\partial_zv+\partial_t Q_{12}\partial_zu+\partial_t( Q_{11}+2Q_{22})\partial_zv))]\\
&+\frac{\partial_zu}{\mathfrak A}[2\partial_zu\partial_zv\cdot ((2F_{11}+F_{22})\partial_zu+F_{12}\partial_zv+\partial_t (2Q_{11}+Q_{22})\partial_zu+\partial_t Q_{12}\partial_zv))]\\
&-\frac{\partial_zv}{\mathfrak A}[((\partial_zu)^2+2(\partial_zv)^2))\cdot ((2F_{11}+F_{22})\partial_zu+F_{12}\partial_zv+\partial_t (2Q_{11}+Q_{22})\partial_zu+\partial_t Q_{12}\partial_zv))]\\
&+\frac{\partial_zv}{\mathfrak A}[2\partial_zu\partial_zv\cdot (F_{12}\partial_zu+(F_{11}+2F_{22})\partial_zv+\partial_t Q_{12}\partial_zu+\partial_t( Q_{11}+2Q_{22})\partial_zv))]\\
\end{aligned}
\end{equation}
now we put all the $\partial_t$ terms on the left, and we obtain that
\begin{equation}
\label{scheme8}
\begin{aligned}
&\partial_tQ_{11}+\partial_tQ_{11}\cdot\frac{2(\partial_zu)^4+2(\partial_zu)^2(\partial_zv)^2}{\mathfrak A}+\partial_tQ_{22}\cdot\frac{(\partial_zu)^4-2(\partial_zu)^2(\partial_zv)^2}{\mathfrak A}+\partial_tQ_{12}\cdot\frac{-(\partial_zu)^3(\partial_zv)+2\partial_zu(\partial_zv)^3}{\mathfrak A}\\
&=F_{11}\cdot\frac{\mathfrak A-2(\partial_zu)^4-2(\partial_zu)^2(\partial_zv)^2}{\mathfrak A}+F_{22}\cdot\frac{-2(\partial_zu)^4+2(\partial_zu)^2(\partial_zv)^2)}{\mathfrak A}+F_{12}\cdot\frac{(\partial_zu)^3\partial_zv-2\partial_zu(\partial_zv)^3}{\mathfrak A}
\end{aligned}
\end{equation}
\begin{equation}
\label{scheme9}
\begin{aligned}
&\partial_tQ_{11}\cdot\frac{(\partial_zv)^4-2(\partial_zu)^2(\partial_zv)^2}{\mathfrak A}+\partial_tQ_{22}+\partial_tQ_{22}\cdot\frac{2(\partial_zv)^4+2(\partial_zu)^2(\partial_zv)^2}{\mathfrak A}+\partial_tQ_{12}\cdot\frac{2(\partial_zu)^3\partial_zv-\partial_zu(\partial_zv)^3}{\mathfrak A}\\
&=F_{11}\cdot \frac{-(\partial_zv)^4+2(\partial_zu)^2(\partial_zv)^2}{\mathfrak A}+F_{22}\cdot\frac{\mathfrak A-2(\partial_zv)^4-2(\partial_zu)^2(\partial_zv)^2}{\mathfrak A}+F_{12}\cdot \frac{-2(\partial_zu)^3(\partial_zv)+\partial_zu(\partial_zv)^3}{\mathfrak A}\\
\end{aligned}
\end{equation}
\begin{equation}
\label{scheme10}
\begin{aligned}
&\partial_tQ_{11}\cdot\frac{3\partial_zu(\partial_zv)^3}{\mathfrak A}+\partial_tQ_{22}\cdot\frac{3(\partial_zu)^3\partial_zv}{\mathfrak A}+\partial_tQ_{12}+\partial_tQ_{12}\cdot\frac{2(\partial_zu)^4+2(\partial_zv)^4-2(\partial_zu)^2(\partial_zv)^2}{\mathfrak A}\\
&+F_{11}\cdot\frac{-3\partial_zu(\partial_zv)^3}{\mathfrak A}+F_{22}\cdot\frac{-3(\partial_zu)^3\partial_zv}{\mathfrak A}+F_{12}\cdot\frac{\mathfrak A-2(\partial_zu)^4-2(\partial_zv)^4+2(\partial_zu)^2(\partial_zv)^2}{\mathfrak A}
\end{aligned}
\end{equation}
and
\begin{equation}
\label{thetaextra1}
(Q_{11}+\theta^2Q_{22}-\theta Q_{12})^2=\left(\frac{2\theta^4+\theta^2+2}{2-\theta^2}\right)^2Q^2_{11}=\left(\frac{2\theta^4+\theta^2+2}{2\theta^2-1}\right)^2Q^2_{22}=\left(\frac{2\theta^4+\theta^2+2}{3\theta}\right)^2Q^2_{12}
\end{equation}
 and
\[
\tilde {\mathfrak C}_1:=\frac{2(\partial_zu)^4+2(\partial_zu)^2(\partial_zv)^2}{\mathfrak A}=\frac{2\theta^4+2\theta^2}{2\theta^4+\theta^2+2}, \quad\tilde {\mathfrak C}_2:= \frac{(\partial_zu)^4-2(\partial_zu)^2(\partial_zv)^2}{\mathfrak A}=\frac{\theta^4-2\theta^2}{2\theta^4+\theta^2+2},\]\[     \tilde {\mathfrak C}_3:= \frac{-(\partial_zu)^3(\partial_zv)+2\partial_zu(\partial_zv)^3}{\mathfrak A}=\frac{-\theta^3+2\theta}{2\theta^4+\theta^2+2},\quad\tilde {\mathfrak C}_4:=\frac{(\partial_zv)^4-2(\partial_zu)^2(\partial_zv)^2}{\mathfrak A}=\frac{1-2\theta^2}{2\theta^4+\theta^2+2},\]\[ \tilde {\mathfrak C}_5:= \frac{2(\partial_zv)^4+2(\partial_zu)^2(\partial_zv)^2}{\mathfrak A}=\frac{2+2\theta^2}{2\theta^4+\theta^2+2}, \quad     \tilde {\mathfrak C}_6:=\frac{2(\partial_zu)^3(\partial_zv)-\partial_zu(\partial_zv)^3}{\mathfrak A}=\frac{2\theta^3-\theta}{2\theta^4+\theta^2+2}
\]
\[
\tilde {\mathfrak C}_7:=\frac{3\partial_zu(\partial_zv)^3}{\mathfrak A}=\frac{3\theta}{2\theta^4+\theta^2+2}, \quad\tilde {\mathfrak C}_8:= \frac{3(\partial_zu)^3\partial_zv}{\mathfrak A}=\frac{3\theta^3}{2\theta^4+\theta^2+2},\]\[ \tilde {\mathfrak C}_9:=\frac{2(\partial_zu)^4+2(\partial_zv)^4-2(\partial_zu)^2(\partial_zv)^2}{\mathfrak A}=\frac{2\theta^4-2\theta^2+2}{2\theta^4+\theta^2+2}
\]
and \eqref{scheme8}-\eqref{scheme10} are equivalent to
\begin{equation}
\label{scheme11}
\partial_t\left\{(1+\tilde {\mathfrak C}_1)Q_{11}+\tilde {\mathfrak C}_2Q_{22}+\tilde {\mathfrak C}_3Q_{12}\right\}=(1-\tilde {\mathfrak C}_1)F_{11}-\tilde {\mathfrak C}_2F_{22}-\tilde {\mathfrak C}_3F_{12}:=\mathfrak R_1
\end{equation}
\begin{equation}
\label{scheme12}
\partial_t\left\{\tilde {\mathfrak C}_4Q_{11}+(1+\tilde {\mathfrak C}_5)Q_{22}+\tilde {\mathfrak C}_6Q_{12}\right\}=-\tilde {\mathfrak C}_4F_{11}+(1-\tilde {\mathfrak C}_5)F_{22}-\tilde {\mathfrak C}_6F_{12}:=\mathfrak R_2
\end{equation}
\begin{equation}
\label{scheme13}
\partial_t\left\{\tilde {\mathfrak C}_7Q_{11}+\tilde {\mathfrak C}_8Q_{22}+(1+\tilde {\mathfrak C}_9)Q_{12}\right\}=-\tilde {\mathfrak C}_7F_{11}-\tilde {\mathfrak C}_8F_{22}+(1-\tilde {\mathfrak C}_9)F_{12}:=\mathfrak R_3
\end{equation}
remind for $\theta=\frac{\partial_zu}{\partial_zv}$, we have \[2\theta Q_{11}+\theta Q_{22}+Q_{12}=0, \quad Q_{11}+2Q_{22}+\theta Q_{12}=0.\] Thus we can solve 
\begin{equation}
\label{thetaextra0}
Q_{22}=\frac{2\theta^2-1}{2-\theta^2}Q_{11},\quad Q_{12}=-\frac{3\theta}{2-\theta^2}Q_{11}, \quad Q_{11}=\frac{2-\theta^2}{2\theta^2-1}Q_{22}, \quad Q_{12}=-\frac{3\theta}{2\theta^2-1}Q_{22}
\end{equation}
from direct calculations, $\tilde {\mathfrak C}_1+\frac{2\theta^2-1}{2-\theta^2}\tilde {\mathfrak C}_2-\frac{3\theta}{2-\theta^2}\tilde {\mathfrak C}_3=0$. Together with \eqref{thetaextra0}, we have $(1+\tilde {\mathfrak C}_1)Q_{11}+\tilde {\mathfrak C}_2Q_{22}+\tilde {\mathfrak C}_3Q_{12}=Q_{11}$. Similarly, $\tilde {\mathfrak C}_4Q_{11}+(1+\tilde {\mathfrak C}_5)Q_{22}+\tilde {\mathfrak C}_6Q_{12}=Q_{22}, \tilde {\mathfrak C}_7Q_{11}+\tilde {\mathfrak C}_8Q_{22}+(1+\tilde {\mathfrak C}_9)Q_{12}=Q_{12}$. Next, we denote \begin{equation}
\label{theta}
\tilde\theta_1=2\theta^2-1,\quad \tilde\theta_2=\theta^2(2-\theta^2),\quad\tilde \theta_3=\frac 13(2-\theta^2)(2\theta^2-1)
\end{equation}
denote $\mathfrak Q_1:=Q_{11}, \mathfrak Q_2:=Q_{22}, \mathfrak Q_3:=Q_{12}$ for convenience. Then
\begin{equation}
\label{scheme18}
\begin{aligned}
&\frac 12\frac{d}{dt}\sum_{i=1}^3\int \tilde\theta_i (\mathfrak Q_i)^2=\sum_{i=1}^3\int \tilde\theta_i\mathfrak Q_i\mathfrak R_i
+\sum_{i=1}^3\frac{\partial_t\tilde\theta_i(\mathfrak Q_i)^2}2:=I_1+I_2
\end{aligned}
\end{equation}
notice that $\tilde{\mathfrak C}_i, \partial_t\tilde{\mathfrak C}_i, \tilde\theta_j, \partial_t\tilde\theta_j$ are uniformly bounded for $1\le i\le 9, 1\le j\le 3$, from Cauchy-Schwartz inequality,  $I_2\lesssim \int Q_{11}^2+Q_{22}^2+Q_{12}^2$. Next we come to estimate $I_1$. Define the operator $\mathfrak T:=u\partial_x+v\partial_y+w\partial_z$. Then
\begin{equation}
\label{scheme21}
F_{11}=\underbrace{-u\partial_xQ_{11}-v\partial_yQ_{11}-w\partial_zQ_{11}}_{\mathfrak TQ_{11}}+\underbrace{\partial_z^2Q_{11}-\left(aQ_{11}-b(Q^2_{11}+Q^2_{12}-\frac 13\mbox{trl}(Q^2))+cQ_{11}\mbox{trl}(Q^2)\right)}_{\mathfrak D_{11}}\underbrace{+\omega Q_{12}}_{\mathfrak V_{11}}
\end{equation}
\begin{equation}
\label{scheme22}
F_{22}=\underbrace{-u\partial_xQ_{22}-v\partial_yQ_{22}-w\partial_zQ_{22}}_{\mathfrak TQ_{22}}+\underbrace{\partial_z^2Q_{22}-\left(aQ_{22}-b(Q^2_{12}+Q^2_{22}-\frac 13\mbox{trl}(Q^2))+cQ_{22}\mbox{trl}(Q^2)\right)}_{\mathfrak D_{22}}\underbrace{-\omega Q_{12}}_{\mathfrak V_{22}}
\end{equation}
\begin{equation}
\label{scheme23} 
F_{12}=\underbrace{-u\partial_xQ_{12}-v\partial_yQ_{12}-w\partial_zQ_{12}}_{\mathfrak TQ_{12}}+\underbrace{\partial_z^2Q_{12}-\left(aQ_{12}-bQ_{12}(Q_{11}+Q_{22})+cQ_{12}\mbox{trl}(Q^2)\right)}_{\mathfrak D_{12}}\underbrace{-\frac 12\omega (Q_{11}-Q_{22})}_{\mathfrak V_{12}}
\end{equation}
remind that $\partial_xu+\partial_yv+\partial_zw=0$. For any functions $f, \Lambda$, from integration by parts, we have \[\int \Lambda \mathfrak Tf=-\int (\partial_x\Lambda\cdot u+\partial_y\Lambda\cdot v+\partial_z\Lambda\cdot w)f,\quad\int \Lambda \mathfrak Tf\cdot f=-\frac 12\int (\partial_x\Lambda \cdot u+\partial_y\Lambda\cdot v+\partial_z\Lambda\cdot w)f^2\] From direct calculations, $\tilde\theta_1\tilde{\mathfrak C_2}=\tilde\theta_2\tilde{\mathfrak C_4}:=\Theta_1,\quad \tilde\theta_3\tilde{\mathfrak C_7}=\tilde\theta_1\tilde{\mathfrak C_3}:=\Theta_2,\quad \tilde\theta_2\tilde{\mathfrak C_6}=\tilde\theta_3\tilde{\mathfrak C_8}:=\Theta_3$. Thus 
\begin{equation}
\label{scheme24}
\begin{aligned}
I_{1.1}:=&\int\tilde\theta_1 Q_{11}((1-\tilde {\mathfrak C}_1)\mathfrak TQ_{11}-\tilde {\mathfrak C}_2\mathfrak TQ_{22}-\tilde {\mathfrak C}_3\mathfrak TQ_{12}) 
+\int\tilde\theta_2Q_{22} (-\tilde {\mathfrak C}_4\mathfrak TQ_{11}+(1-\tilde {\mathfrak C}_5)\mathfrak TQ_{22}-\tilde {\mathfrak C}_6\mathfrak TQ_{12})\\
+&\int\tilde\theta_3 Q_{12}(-\tilde {\mathfrak C}_7\mathfrak TQ_{11}-\tilde {\mathfrak C}_8\mathfrak TQ_{22}+(1-\tilde {\mathfrak C}_9)\mathfrak TQ_{12})\\
=-&\frac 12\sum_{i=1}^3\int [\partial_x(\tilde\theta_i(1-\tilde {\mathfrak C}_{4i-3}))u+\partial_y(\tilde\theta_i(1-\tilde {\mathfrak C}_{4i-3}))v+\partial_z(\tilde\theta_i(1-\tilde {\mathfrak C}_{4i-3}))w)](\mathfrak Q_i)^2 \\
+2&\int(\partial_x\Theta_1 u+\partial_y\Theta_1 v+\partial_z\Theta_1 w)Q_{11}Q_{22}+ (\partial_x\Theta_2 u+\partial_y\Theta_2 v+\partial_z\Theta_2 w)Q_{11}Q_{12}+ (\partial_x\Theta_3 u+\partial_y\Theta_3 v+\partial_z\Theta_3 w)Q_{22}Q_{12}\\
\end{aligned}
\end{equation}
notice that $\partial(\tilde\theta_i(1-\tilde {\mathfrak C}_{4i-3})), \partial\Theta_i\sim \partial\theta$ for $1\le i\le 3$, from $\partial_x\theta\cdot u+\partial_y\theta\cdot v+\partial_z\theta\cdot w$ is bounded, we have $I_{1.1}\lesssim \int Q_{11}^2+Q_{12}^2+Q_{22}^2$. The terms includng $\mathfrak V_{11}, \mathfrak V_{12}, \mathfrak V_{22}$ can be controlled by $\int Q_{11}^2+Q_{22}^2+Q_{12}^2$ because $\omega$ is uniformly bounded.\\Next, for the diffusion terms of $I_1$, which are $\mathfrak D_{11}, \mathfrak D_{22}, \mathfrak D_{12}$, we start by studying the quadratic form
\begin{equation}
\label{scheme25}
\begin{aligned}
Q(x,y,z):=&\tilde \theta_1(1-\tilde{\mathfrak C_1})x^2+\tilde \theta_2(1-\tilde{\mathfrak C_5})y^2+\tilde \theta_3(1-\tilde{\mathfrak C_9})z^2-2\Theta_1xy-2\Theta_2xz-2\Theta_3yz
\end{aligned}
\end{equation}
it is equivalent to $(x,y,z)A(x,y,z)^T$, where the matrix $A$ has the form
\begin{equation}
\label{matrix1}
A=
\begin{pmatrix}
\tilde \theta_1(1-\tilde{\mathfrak C_1}) & -\Theta_1&-\Theta_2\\
-\Theta_1 & \tilde \theta_2(1-\tilde{\mathfrak C_5})&-\Theta_3\\
-\Theta_2 & -\Theta_3& \tilde \theta_3(1-\tilde{\mathfrak C_9})\\
\end{pmatrix}
=\frac{(2-\theta^2)(2\theta^2-1)}{2\theta^4+\theta^2+2}\cdot\begin{pmatrix}
1&\theta^2&-\theta\\
\theta^2&\theta^4&-\theta^3\\
-\theta&-\theta^3&\theta^2
\end{pmatrix}
\end{equation}
which means $Q(x,y,z)=\frac{(2-\theta^2)(2\theta^2-1)}{2\theta^4+\theta^2+2}(x+\theta^2y-\theta z)^2$. We divide the values of $\theta$ into two cases.\\
\underline{\textbf{Case 1: $\theta^2\in (\frac 12,2)$.}} Then $\tilde\theta_1, \tilde\theta_2, \tilde\theta_3>0, (2-\theta^2)(2\theta^2-1)>0$. So for the term $aQ_{ij}$ in $\mathfrak D_{ij}$, it equals to
\[
-a\frac{(2-\theta^2)(2\theta^2-1)}{2\theta^4+\theta^2+2}\int (Q_{11}+\theta^2Q_{22}-\theta Q_{12})^2
\lesssim_{\theta}-a\int Q_{11}^2+Q_{22}^2+Q_{12}^2
\]
similarly for the term   $cQ_{ij}\text{tr}(Q^2)$ in $\mathfrak D_{ij}$, remind that $\text{tr}(Q^2)\gtrsim Q_{11}^2+Q_{22}^2+Q_{12}^2$, so it equals to
\begin{equation*}
\begin{aligned}
-c\frac{(2-\theta^2)(2\theta^2-1)}{2\theta^4+\theta^2+2}\int (Q_{11}+\theta^2Q_{22}-\theta Q_{12})^2\text{tr}(Q^2)
&\lesssim_{\theta}-c\int( Q_{11}^2+Q_{22}^2+Q_{12}^2)^2\lesssim-c\int Q_{11}^4+Q_{22}^4+Q_{12}^4
\end{aligned}
\end{equation*}
 for the term $bQ_{ij}Q_{i'j'}$ in $\mathfrak D_{ij}$, from Holder inequality, it can be controlled by \[|b|\left(\int Q_{11}^2+Q_{22}^2+Q_{12}^2 \right)^\frac 12\left( \int( Q_{11}^4+Q_{22}^4+Q_{12}^4) \right)^\frac 12.\] Finally, we deal with the term $\partial_z^2Q_{ij}$ in $\mathfrak D_{ij}$. From integration by parts, and similarl to \eqref{scheme25}, it can be written as 
 \begin{equation}
 \label{scheme26}
 \begin{aligned}
 \underbrace{-\nu_2\int (\partial_z Q_{11}, \partial_z Q_{22}, \partial_z Q_{12})A (\partial_z Q_{11}, \partial_z Q_{22}, \partial_z Q_{12})^T}_{\text{Term 1}}+\text{other terms}
 \end{aligned}
 \end{equation}
 Similarly as before, Term 1 equals to $-\frac{(2-\theta^2)(2\theta^2-1)}{2\theta^4+\theta^2+2}\int (\partial_zQ_{11}+\theta^2\partial_zQ_{22}-\theta \partial_zQ_{12})^2$. Take $\partial_z$ of \eqref{thetaextra0}, we have
 \begin{equation*}
 \begin{aligned}
 &\partial_zQ_{11}+\theta^2\partial_zQ_{22}-\theta \partial_zQ_{12}\\=&\frac{2\theta^4+\theta^2+2}{2-\theta^2}\partial_zQ_{11}+p_1(\theta)\partial_z\theta Q_{11}
 =\frac{2\theta^4+\theta^2+2}{2\theta^2-1}\partial_zQ_{22}+p_2(\theta)\partial_z\theta Q_{22}=-\frac{2\theta^4+\theta^2+2}{3\theta}\partial_zQ_{12}-p_3(\theta)\partial_z\theta Q_{12}
 \end{aligned}
 \end{equation*}
 where $p_1(\theta), p_2(\theta), p_3(\theta)$ are uniformly bounded in $\theta$. So there exist constants $a(\theta), b(\theta)>0$, such that
 \begin{equation*}
 \begin{aligned}
 \text{Term 1}&\le -a(\theta)\int |\partial_zQ_{11}|^2+|\partial_zQ_{22}|^2+|\partial_zQ_{12}|^2+b(\theta)\int|\partial_z\theta|^2(Q^2_{11}+Q_{22}^2+Q_{12}^2)
 \end{aligned}
 \end{equation*}
 remind that we suppose $|\partial_z\theta|$ is uniformly bounded in $z\in [0,1]$. Thus if we choose $a\gg 1$, then $b(\theta)|\partial_z\theta|^2$ can be controlled by $a$. The estimate of Term 1 is thus finished. For the estimate of other terms, notice that they all have the form like $\int_{\mathbb R^2\times [0,1]}\partial_z\tilde \theta_{k}\partial_zQ_{ij}Q_{i'j'} d\textbf{x}$ or $ \int_{\mathbb R^2\times [0,1]}\partial_z\mathfrak C_{l}\partial_zQ_{ij}Q_{i'j'} d\textbf{x}$, where $1\le k\le 3, 1\le l\le 9, 1\le i,j,i',j'\le 2$. Because $\partial_z\tilde \theta_{k}, \partial_z\mathfrak C_{l}$ can be written as $\partial_z\theta\times $ uniformly bounded functions of $\theta$, from Holder inequality, they can all be controlled by $\left(\int |\partial_zQ_{11}|^2+|\partial_zQ_{22}|^2+|\partial_zQ_{12}|^2\right)^\frac 12\left(\int Q^2_{11}+Q_{22}^2+Q_{12}^2\right)^\frac 12$. \\
 Combine all the estimates above together, we have
 \[\frac 12\frac{d}{dt}\int\tilde\theta_1 Q_{11}^2+\tilde\theta_2Q_{22}^2+\tilde\theta_3\mathfrak Q_{12}^2
\lesssim-(\nu_2-\varepsilon_0)\int{ |\partial_zQ|^2}-\int Q^2-\int Q^4.\] From Gronwall inequality, there exist constants $C, \lambda$ depending on $a,b,c, c_1,c_2, \varepsilon$, such that for any $t>0$,
\[
\int Q^2(t)+(\nu_2-\varepsilon)\int_0^t{e^{s-t}\int |\partial_zQ|^2(s)ds}+\int_0^t{e^{s-t}\int Q^4(s)ds}\le Ce^{-\lambda t}
\]
\underline{\textbf{Case 2:  If $\theta^2\in (0,\frac 12)\cup (2, \infty)$.}}
 From direct calculations, $\tilde\theta_1Q_{11}^2+\tilde\theta_2Q_{2}^2+\tilde\theta_3Q_{12}^2=\frac{(2\theta^2-1)(2+3\theta^2)}{2-\theta^2}Q_{11}^2$, where $\frac{(2\theta^2-1)(2+3\theta^2)}{2-\theta^2}<0$ because  $\theta^2\in (0,\frac 12)\cup (2, \infty)$. So in fact, the integral on the left is the same order as $-\frac d{dt}\int Q_{11}^2+Q_{22}^2+Q_{12}^2$. Moreover, the matrix $A$ in \eqref{matrix1} is negative instead of positive. So we just need to change all the symbols of $+$ and $-$, and we obtain the similar result.
 
 Next we come to calculate the $H^1$ norm of $Q_{11}, Q_{12}, Q_{22}$. Remind \eqref{scheme11}-\eqref{scheme13}. For any $\partial$ belongs to $\{\partial_x, \partial_y,\partial_z\}$, we have
 \[
\frac 12\frac d{dt}\sum_{i=1}^3\int \tilde\theta_i|\nabla \mathfrak Q_i|^2=\sum_{i=1}^3\int \tilde\theta_i\nabla \mathfrak Q_i\nabla \mathfrak R_i+\frac 12\sum_{i=1}^3\int \partial_t\tilde\theta_i|\nabla \mathfrak Q_i|^2:=\mathfrak I_1+\mathfrak I_2
 \]
 similarly as before, $\mathfrak I_2\lesssim \sum_i\int |\nabla Q_{ij}|^2$.  Next, notice that
 \[
 \nabla F_{11}=\mathfrak T(\nabla Q_{11})\underbrace{-\nabla u\partial_xQ_{11}-\nabla v \partial_y Q_{11}-\nabla w\partial_zQ_{11}}_{\tilde{\mathfrak T}_{11}}-\nabla \mathfrak D_{11}\underbrace{-\nabla \omega Q_{12}-\omega\nabla Q_{12}}_{\tilde{\mathfrak V}_{11}}
 \]
 we also have similar equations for $\nabla F_{12}, \nabla F_{22}$. The terms including $\mathfrak T(\nabla Q_{ij})$ can be estimated as \eqref{scheme24}, which can be controlled by $\int |\nabla Q_{ij}|^2$. Moreover, remind that $\nabla\tilde{\textbf{u}}$ and $\nabla\omega$ are uniformly bounded, the integrals including $\tilde{\mathfrak T}$ and $\tilde{\mathfrak V}$ are all controlled by $ \sum_i\int |\nabla Q_{ij}|^2$. Finally we estimate the integrals including $\nabla\mathfrak D$. Similarly as before, the terms including $a,b,c$ can be finally controlled by $-a\sum_i\int |\nabla Q_{ij}|^2 -c\sum_i\int |\nabla (Q^2_{ij})|^2 $. For the terms including $\partial_z^2\nabla Q_{ij}$, from the boundary condition of $Q_{ij}$, we integrate by parts to get that it can be written as 
 \[
  {- \nu_2\int (\partial_z \nabla Q_{11}, \partial_z \nabla Q_{22}, \partial_z \nabla Q_{12})A (\partial_z \nabla Q_{11}, \partial_z \nabla Q_{22}, \partial_z\nabla  Q_{12})^T}+\text{other terms}
 \]
 which can be controlled by
 \[
 -a'(\theta)\int {|\partial_z \nabla Q_{ij}|^2}+ b'(\theta)\int {|\partial_z \nabla Q_{ij}|^2}
 \]
 here $a'(\theta), b'(\theta)$ only depend on $\theta, \nabla\theta$. Thus if we choose $a,c$ large enough,  we have  $\frac d{dt}\int |\nabla Q_{ij}|^2\lesssim -(\nu_2-\varepsilon_0)\int |\partial_z\nabla Q_{ij}|^2-\int |\nabla Q_{ij}|^2-\int |Q_{ij}\nabla Q_{ij}|^2$, no matter $\theta^2\in (\frac 12,2)$ or $\theta^2\in (0,\frac 12)\cup (2,\infty)$, and we obtain the same result by using Gronwall inequality. This finishes the proof of Lemma \ref{convergenceQ1}.
\end{proof}
Moreover, for $\theta^2\in (\frac 12, 2)$, from the analysis before, for any $\nu_2, c>0, b\in\mathbb R$, if we choose $a$ large enough, then there exist constants $\lambda_i>0 (1\le i\le 3)$ that do not depend on $\theta$, such that
\begin{equation}\label{tougha1}
\begin{aligned}
&\frac{d}{dt}\int {\tilde\theta_1 |\nabla Q_{11}|^2+ \tilde\theta_2 |\nabla Q_{22}|^2+\tilde\theta_3 |\nabla Q_{12}|^2}+\nu_2\lambda_1||\partial_z(\nabla Q_{11}, \nabla Q_{22}, \nabla Q_{12})||^2_{L^2}\\+&a \lambda_2||(\nabla Q_{11}, \nabla Q_{22}, \nabla Q_{12})||^2_{L^2}+ c\lambda_3||(Q_{11}, Q_{22}, Q_{12})\otimes (\nabla Q_{11}, \nabla Q_{22}, \nabla Q_{12})||^2_{L^2}\le 0
\end{aligned}
\end{equation}
which equals to
\begin{equation}\label{tougha2}
\begin{aligned}
&||\tilde\theta_1 |\nabla Q_{11}|^2+ \tilde\theta_2 |\nabla Q_{22}|^2+\tilde\theta_3 |\nabla Q_{12}|^2||_{L^\infty(0,t)L^1(\mathcal S)}\\+&\nu_2\lambda_1\int_0^t||\partial_z(\nabla Q_{11}(s,\cdot), \nabla Q_{22}(s,\cdot), \nabla Q_{12}(s,\cdot))||^2_{L^2(\mathcal S)}+a \lambda_2\int_0^t||(\nabla Q_{11}(s,\cdot), \nabla Q_{22}(s,\cdot), \nabla Q_{12}(s,\cdot))||^2_{L^2(\mathcal S)}\\+& c\lambda_3\int_0^t||(Q_{11}(s,\cdot), Q_{22}(s,\cdot), Q_{12}(s,\cdot))\otimes (\nabla Q_{11}(s,\cdot), \nabla Q_{22}(s,\cdot), \nabla Q_{12}(s,\cdot))||^2_{L^2(\mathcal S)}\\\le& ||\tilde\theta_1 |\nabla Q_{11}(0,\cdot)|^2+ \tilde\theta_2 |\nabla Q_{22} (0,\cdot)|^2+\tilde\theta_3 |\nabla Q_{12}(0,\cdot)|^2||_{L^1(\mathcal S)}.
\end{aligned}
\end{equation}
The following lemma gives the result about $Q_{13}, Q_{23}$.
\begin{lem}
\label{lemtough1}
Under the assumptions of Theorem \ref{convergenceQ1}, we denote the functions
\begin{equation}
\label{toughb0}
\Upsilon_1:=\frac{{\textbf u}\cdot\theta}{\partial_zv}, \quad\Upsilon_2:=\frac{\omega}{\partial_zv},\quad \Upsilon_3:=\frac{\partial_z\theta}{\partial_zv},  \quad\Upsilon_4:=\frac{\partial_z^2\theta}{\partial_zv},\quad  \Upsilon_5:=\frac{\partial_5\theta}{\partial_zv}
\end{equation}
if $\Upsilon_i (1\le i\le 5)$ are uniformly bounded for any $t\ge 0, (x,y,z)\in\mathcal S$, then $Q_{13}, Q_{23}$ are $H^1$ bounded and have exponential decay of $L^2$ norm, i.e, there exist constants $\tilde C_1,\tilde C_2, \tilde\lambda>0 $, such that
\[
\int Q_{13}^2(t,\cdot)+Q_{23}^2(t,\cdot) dt\le \tilde C_1e^{-\tilde\lambda t}, \quad \int_0^\infty\int |\nabla Q_{13}^2(t,\cdot)|+|\nabla Q_{23}^2(t,\cdot)| dt\le \tilde C_2.
\]
moreover, if we suppose $\partial_z\theta\equiv 0$, i.e. there exist functions $\mathsf g_1(x,y), \mathsf g_2(x,y)$, such that $u=\mathsf g_1v+\mathsf g_2$, then $Q_{13}, Q_{23}$ have the same behaviour of $H^1$ convergence as $(Q_{11}, Q_{12},Q_{22})$.
\end{lem}
\begin{rmk}
A natural example is that there exists a constant $c\ne 0, \pm\frac {\sqrt 2}2, \pm\sqrt 2$, such that $\theta\equiv c$. Then  $\partial_zu=c\partial_zv$, which menas that $u=cv+f(x,y)$. Together with $\eqref{Prandtl3}_1, \eqref{Prandtl3}_2$, we have $c\partial_xf+\partial_yf=0$, which means that $f(x,y)=f_0(x-cy)$. 
\end{rmk}
\begin{proof}
We need to study the terms $F_{11}, F_{12}, F_{22}$ more carefully. Notice that
\begin{equation}
\label{toughb1}
\begin{cases}
F_{11}=-\textbf{u}\cdot\nabla Q_{11}
+Q_{12}\omega+\nu_2\partial_z^2Q_{11}-\left(aQ_{11}-b(Q^2_{11}+Q^2_{12}-\frac 13\mbox{trl}(Q^2))+cQ_{11}\mbox{trl}(Q^2)\right)\\
F_{12}=-\textbf{u}\cdot\nabla Q_{12}
-\frac 12(Q_{11}-Q_{22})\omega+ \nu_2\partial_z^2Q_{12}-\left(aQ_{12}-bQ_{12}(Q_{11}+Q_{22})+cQ_{12}\mbox{trl}(Q^2)\right)\\
F_{22}=-\textbf{u}\cdot\nabla Q_{22}-Q_{12}\omega+\nu_2\partial_z^2Q_{22}-\left(aQ_{22}-b(Q^2_{12}+Q^2_{22}-\frac 13\mbox{trl}(Q^2))+cQ_{22}\mbox{trl}(Q^2)\right)\\
\end{cases}
\end{equation}
we denote $\Xi^1:=-\frac{3\theta}{2-\theta^2}, \Xi^2=\frac{2\theta^2-1}{2-\theta^2}$ for convenience. Then $Q_{12}=\Xi_1Q_{11}, Q_{22}=\Xi_2 Q_{11}$.
It is not difficult to verify that 
\[
aQ_{12}-bQ_{12}(Q_{11}+Q_{22})+cQ_{12}\mbox{trl}(Q^2)=\Xi_1\left(aQ_{11}-b(Q^2_{11}+Q^2_{12}-\frac 13\mbox{trl}(Q^2))+cQ_{11}\mbox{trl}(Q^2)\right)
\]
\[
aQ_{22}-b(Q^2_{12}+Q^2_{22}-\frac 13\mbox{trl}(Q^2))+cQ_{22}\mbox{trl}(Q^2)=\Xi_2\left(aQ_{11}-b(Q^2_{11}+Q^2_{12}-\frac 13\mbox{trl}(Q^2))+cQ_{11}\mbox{trl}(Q^2)\right)
\]
next,
\[
\partial_z^2Q_{12}=\Xi_1\partial_z^2Q_{11}+2\partial_z\Xi_1\partial_zQ_{11}+\partial_z^2\Xi_1Q_{11},\quad \partial_z^2Q_{22}=\Xi_2\partial_z^2Q_{11}+2\partial_z\Xi_2\partial_zQ_{11}+\partial_z^2\Xi_2Q_{11},
\]
\[
-\textbf{u}\cdot\nabla Q_{12}=-\Xi_1\textbf{u}\cdot \nabla Q_{11}-(\textbf u\cdot \nabla \Xi_1)Q_{11},\quad  -\textbf{u}\cdot\nabla Q_{22}=-\Xi_2\textbf{u}\cdot \nabla Q_{11}-(\textbf u\cdot \nabla \Xi_2)Q_{11}
\]
\[
Q_{12}\omega=\Xi_1\omega Q_{11}, \quad-\frac 12(Q_{11}-Q_{22})\omega=-\frac {1-\Xi_2}{2}\omega Q_{11}\quad -Q_{12}\omega=-\Xi_1\omega Q_{11}
\]
notice that
\[
\nabla\Xi_1=-\frac{3\theta^2+6}{(2-\theta^2)^2}\nabla\theta, \quad\nabla \Xi_2=\frac{6\theta}{(2-\theta^2)^2}\nabla\theta
\]
\[\partial_z^2\Xi_1=-\frac{3\theta^2+6}{(2-\theta^2)^2}\partial_z^2\theta-\frac{6\theta^3+36\theta}{(2-\theta^2)^3}(\partial_z\theta)^2, \quad\partial_z^2\Xi_2=\frac{6\theta}{(2-\theta)^2}\partial_{z}^2\theta+\frac{18\theta^2+12}{(2-\theta^2)^3}(\partial_z\theta)^2
\]
so we have
\begin{equation}
\begin{aligned}
F_{12}&=\Xi_1F_{11}+\frac{3\theta^2+6}{(2-\theta^2)^2}(\textbf u\cdot \nabla \theta)Q_{11}-\left(\Xi_1^2+\frac{1-\Xi_2}{2}\right)\omega\\&-\frac{(6\theta^2+12)\nu_2}{(2-\theta^2)^2}\partial_z\theta\partial_zQ_{11}-\left(\frac{(3\theta^2+6)\nu_2}{(2-\theta^2)^2}\partial_z^2\theta+\frac{(6\theta^3+36\theta)\nu_2}{(2-\theta^2)^3}(\partial_z\theta)^2\right)Q_{11}
\end{aligned}
\end{equation}
and
\begin{equation}
\begin{aligned}
F_{22}&=\Xi_1F_{11}-\frac{6\theta}{(2-\theta^2)^2}(\textbf u\cdot \nabla \theta)Q_{11}-\left(\Xi_1\Xi_2+\Xi_1\right)\omega\\&+\frac{12\theta\nu_2}{(2-\theta^2)^2}\partial_z\theta\partial_zQ_{11}+\left(\frac{(6\theta)\nu_2}{(2-\theta^2)^2}\partial_z^2\theta+\frac{(18\theta^2+12)\nu_2}{(2-\theta^2)^3}(\partial_z\theta)^2\right)Q_{11}
\end{aligned}
\end{equation}
recall \eqref{scheme1}. We have
\[
Q_{13}=\frac{(2\theta^3+2\theta)F_{11}+(2-\theta^2)F_{12}+(\theta^3-2\theta)F_{22}}{(2\theta^4+\theta^2+2)\partial_zv}-\frac{(9\theta^2+6)\partial_t\theta }{(2-\theta^2)(2\theta^4+\theta^2+2)\partial_zv}Q_{11}
\]
and
\[
Q_{23}=\frac{(1-2\theta^2)F_{11}+(2\theta^3-\theta)F_{12}+(2\theta^2+2)F_{22}}{(2\theta^4+\theta^2+2)\partial_zv}+\frac{(6\theta^4+9\theta^2)\partial_t\theta }{(2-\theta^2)(2\theta^4+\theta^2+2)\partial_zv}Q_{11}
\]
from direct calculations, $2\theta^3+2\theta+(2-\theta^2)\Xi_1+(\theta^3-2\theta)\Xi_2=1-2\theta^2+(2\theta^3-\theta)\Xi_1+(2\theta^2+2)\Xi_2=0$.
Moreover,
\begin{equation}
\label{toughb5}
\begin{aligned}
Q_{13}&=\frac{1}{(2\theta^4+\theta^2+2)\partial_zv}\cdot\left(\frac{9\theta^2+6}{2-\theta^2}(\textbf {u}\cdot\nabla\theta) Q_{11}+\frac{3\theta^6-9\theta^4+21\theta^2-12}{2(2-\theta^2)^2}\omega Q_{11}-\frac{(18\theta^2+12)\nu_2}{(2-\theta^2)^2}\partial_z\theta\partial_zQ_{11}\right)\\
&+\frac{1}{(2\theta^4+\theta^2+2)\partial_zv}\left(-\frac{(9\theta^2+6)\nu_2}{2-\theta^2} \partial_z^2\theta-\frac{(24\theta^3+48\theta)\nu_2}{(2-\theta^2)^2}(\partial_z\theta)^2-\frac{9\theta^2+6}{2-\theta^2}\partial_t\theta\right)Q_{11}
\end{aligned}
\end{equation}
\begin{equation}
\label{toughb6}
\begin{aligned}
Q_{23}&=\frac{1}{(2\theta^4+\theta^2+2)\partial_zv}\cdot\left(-\frac{6\theta^3+9\theta}{2-\theta^2}(\textbf {u}\cdot\nabla\theta) Q_{11}-\frac{3\theta(\theta^2+1)(2\theta^2+3)}{2-\theta^2}\omega Q_{11}+\frac{(12\theta^3+18\theta)\nu_2}{(2-\theta^2)^2}\partial_z\theta\partial_zQ_{11}\right)\\
&+\frac{1}{(2\theta^4+\theta^2+2)\partial_zv}\left(-\frac{(6\theta^3+9\theta)\nu_2}{2-\theta^2} \partial_z^2\theta-\frac{(12\theta^6+102\theta^4+24\theta^2+24)\nu_2}{(2-\theta^2)^2}(\partial_z\theta)^2+\frac{6\theta^4+9\theta^2}{2-\theta^2}\partial_t\theta\right)Q_{11}
\end{aligned}
\end{equation}
if $\partial_z\theta\equiv 0$, then $Q_{13}, Q_{23}\sim Q_{11}$. Otherwise $Q_{13}, Q_{23}\lesssim |Q_{11}|+|\partial_zQ_{11}|$. The lemma is proved from the result of Theorem \ref{convergenceQ1}.
\end{proof}
Moreover, we prove a lemma about the $L^2$ norm of $\partial_tQ_{11}, \partial_t\partial_zQ_{11}$, which which be used for proving the convergence of the hydrostatic system towards the anisotropic system.
\begin{lem}
(1)For $\textbf u$ as the solution of \eqref{Prandtl3}, if $u,v,w,\omega, \partial_zu, \Upsilon_i (1\le 5)$ are uniformly bounded for any $t\ge 0, (x,y,z)\in\mathcal S$, then there exists a constant $C>0$, such that for $Q$ as the solution of \eqref{extra1}, we have
\begin{equation}
\label{partialtQ1}
\begin{aligned}
||\partial_tQ_{11}||_{L^2_t(L^2)}+||\partial_zQ_{11}||_{L^\infty_t(L^2)}\le &C(||\partial_zQ_{11}(0,\cdot)||_{L^\infty_t(L^2)}+||Q_{11}||_{L^2_t(H^1)}+||Q^2_{11}||_{L^2_t(L^2)}+||Q^3_{11}||_{L^2_t(L^2)})
\end{aligned}
\end{equation}
(2)For some $j\in\{1,2,3\}$, if $\omega, \partial_ju, \partial_jv, \partial_jw, \partial_j\omega, \partial_z\partial_iu, \partial_j\Upsilon_i (1\le i\le  5)$ are uniformly bounded for any $t\ge 0, (x,y,z)\in\mathcal S$, then there exists a constant $C>0$, such that for $Q$ as the solution of \eqref{extra1}, we have
\begin{equation}
\label{partialtQ1}
\begin{aligned}
&||\partial_t\partial_iQ_{11}||_{L^2_t(L^2)}+||\partial_z\partial_iQ_{11}||_{L^\infty_t(L^2)}\\
&\le C(||\partial_iQ_{11}(0,\cdot)||_{L^\infty_t(L^2)}+||Q_{11}||_{L^2_t(L^2)}+||\partial_iQ_{11}||_{L^2_t(H^1)}+||Q_{11}\partial_iQ_{11}||_{L^2_t(L^2)}+||Q^2_{11}\partial_iQ_{11}||_{L^2_t(L^2)}).
\end{aligned}
\end{equation}
\end{lem}
\begin{proof}
Remind $Q_{11}$ satisfies $\eqref{extra1}_1$. Thus
\begin{equation}
\label{partialtQ1}
\begin{aligned}
||\partial_tQ_{11}||_{L^2}^2&=-(\textbf u\cdot\nabla Q_{11}, \partial_tQ_{11})_{L^2}+(Q_{12}\omega, \partial_tQ_{11})_{L^2}-(Q_{13}\partial_zu, \partial_tQ_{11})_{L^2}+\nu_2(\partial_z^2Q_{11}^2, \partial_tQ_{11})_{L^2}\\&-\left(aQ_{11}-b(Q^2_{11}+Q^2_{12}-\frac 13\mbox{trl}(Q^2))+cQ_{11}\mbox{trl}(Q^2), \partial_{t}Q_{11}\right)_{L^2}
\end{aligned}
\end{equation}
because $(\partial_z^2Q_{11}, \partial_tQ_{11})_{12}=-\frac 12\frac d{dt}||\partial_zQ_{11}||_{L^2}^2$, we have
\begin{equation}
\label{partialtQ2}
\begin{aligned}
||\partial_tQ_{11}||_{L^2_t(L^2)}+||\partial_zQ_{11}||_{L^\infty_t(L^2)}\lesssim &||\partial_zQ_{11}(0,\cdot)||_{L^\infty_t(L^2)}+||\textbf u\cdot\nabla Q_{11}||_{L^2_t(L^2)}+||Q_{12}\omega||_{L^2_t(L^2)}+||Q_{13}\partial_zu||_{L^2_t(L^2)}\\
+&a||Q_{11}||_{L^2_t(L^2)}+b||Q^2_{11}||_{L^2_t(L^2)}+c||Q^3_{11}||_{L^2_t(L^2)}
\end{aligned}
\end{equation}
similarly,
\begin{equation}
\label{partialtQ3}
\begin{aligned}
&||\partial_t\partial_iQ_{11}||_{L^2_t(L^2)}+||\partial_z\partial_iQ_{11}||_{L^\infty_t(L^2)}\\\lesssim &||\partial_z\partial_iQ_{11}(0,\cdot)||_{L^\infty_t(L^2)}+||\partial_i(\textbf u\cdot\nabla Q_{11})||_{L^2_t(L^2)}+||\partial_i(Q_{12}\omega)||_{L^2_t(L^2)}+||\partial_i(Q_{13}\partial_zu)||_{L^2_t(L^2)}\\
+&a||\partial_iQ_{11}||_{L^2_t(L^2)}+b||Q_{11}\partial_iQ_{11}||_{L^2_t(L^2)}+c||Q^2_{11}\partial_iQ_{11}||_{L^2_t(L^2)}
\end{aligned}
\end{equation}
and the lemma is proved by using the result of Lemma  \ref{lemtough1}.
\end{proof}
\section{Convergence to the hydrostatic system}
\label{section4}
In this section we prove Theorem \ref{thmconvergence1} about the convergence of anisotropic system towards hydrostatic system.
\begin{proof}
Simiarly as \eqref{matrix1} to \eqref{matrix3}, we define the matrix $\f S_1, \f S_2, \f S_3$ of the rotation term $Q_{ij}$, which is
\begin{equation}
\label{matrix4}
\f S_1:=\omega
\begin{pmatrix}
-Q_{12}&
\frac 12(Q_{11}-Q_{22})&
0\\
\frac 12(Q_{11}-Q_{22})&
Q_{12}&
0\\
0&
0&0
\end{pmatrix}
\end{equation}
\begin{equation}
\label{matrix5}
\f S_2:=-\partial_zu
\begin{pmatrix}
-Q_{13}&
-\frac 12Q_{23}&
\frac 1{2\va}( Q_{11}-Q_{33})\\
-\frac 12Q_{23}&
0&
\frac{1}{2\varepsilon}Q_{12}\\
\frac 1{2\va}( Q_{11}-Q_{33})&
\frac{1}{2\varepsilon}Q_{12}&
Q^\va_{13}
\end{pmatrix}
\end{equation}
\begin{equation}
\label{matrxi6}
\quad\f S_3:=-\partial_zv
\begin{pmatrix}
0&
-\frac 12Q_{13}&
\frac 1{2\va}Q_{12}\\
-\frac 12Q_{13}&
-Q_{23}&
\frac{1}{2\varepsilon}(Q_{22}-Q_{33})\\
\frac 1{2\va}Q_{12}&
\frac{1}{2\va}(Q_{22}-Q_{33})&
Q_{23}^\va
\end{pmatrix}
\end{equation}
define $\tilde\omega^\va:=\omega_0^\va-\omega=\partial_x\f v^\va-\partial_y\f u^\va$. We set
\begin{equation}
\begin{aligned}
&\f H^\va(t):=\frac 12||(\f u^\va, \f v^\va, \va \f w^\va) ||_{L^2}^2+\frac 12||(\f Q_{11}^\va, \f Q_{12}^\va, \f Q_{21}^\va, \f Q_{22}^\va, \f Q_{33}^\va ,\va\f Q_{13}^\va, \va \f Q_{23}^\va ,\va\f Q_{31}^\va, \va \f Q_{32}^\va) ||_{L^2}^2\\
+&\frac 12||\partial_\va(\f Q_{11}^\va, \f Q_{12}^\va, \f Q_{21}^\va, \f Q_{22}^\va, \f Q_{33}^\va ,\va\f Q_{13}^\va, \va \f Q_{23}^\va ,\va\f Q_{31}^\va, \va \f Q_{32}^\va) ||_{L^2}^2
\end{aligned}
\end{equation}
Notice that $\f H^\va_0\lesssim \va^2$. We set
\begin{equation}
\begin{aligned}
&T^*:=\sup\{t\ge 0, \text{such that}\,\, \f H^\va(t)\lesssim \va^2\,\,\text{for any}\,\,t\ge 0\}
\end{aligned}
\end{equation}
thus for any $t\in [0, T^*)$, we have
\begin{equation}
\label{toughc1}
\begin{aligned}
&\frac d{dt}\f H^\va(t)+\nu_1||\partial_\va(\f u^\va, \f v^\va, \va \f w^\va) ||_{L^2}^2+\nu_2||\partial_\va(\f Q_{11}^\va, \f Q_{12}^\va, \f Q_{21}^\va, \f Q_{22}^\va, \f Q_{33}^\va ,\va\f Q_{13}^\va, \va \f Q_{23}^\va ,\va\f Q_{31}^\va, \va \f Q_{32}^\va) ||_{L^2}^2\\
&+\nu_2||\partial_\va^2(\f Q_{11}^\va, \f Q_{12}^\va, \f Q_{21}^\va, \f Q_{22}^\va, \f Q_{33}^\va ,\va\f Q_{13}^\va, \va \f Q_{23}^\va ,\va\f Q_{31}^\va, \va \f Q_{32}^\va) ||_{L^2}^2:=\sum_{i=0}^{14}\f T_{i}
\end{aligned}
\end{equation}
where
\begin{equation}
\label{toughdetails1}
\f T_0:=\va^2\nu_1\left[(\Delta_hu, \f u^\va)_{L^2}+(\Delta_hv, \f v^\va)_{L^2}+\va^2(\Delta_hw, \f w^\va)_{L^2}\right]-\va^2(\partial_tw, \f w^\va)_{L^2}
\end{equation}
\begin{equation}
\label{toughdetails2}
\begin{aligned}
\f T_1:=&-(\f u^\va\partial_xu+u^\va\partial_x\f u^\va, \f u^\va)_{L^2}-(\f v^\va\partial_yu+v^\va\partial_y\f u^\va, \f u^\va)_{L^2}-(\f w^\va\partial_zu+w^\va\partial_z\f u^\va, \f u^\va)_{L^2}-(\f u^\va\partial_xv+u^\va\partial_x\f v^\va, \f v^\va)_{L^2}\\&-(\f v^\va\partial_yv+v^\va\partial_y\f v^\va, \f v^\va)_{L^2}-(\f w^\va\partial_zv+w^\va\partial_z\f v^\va, \f v^\va)_{L^2}-\va^2(u^\va\partial_xw^\va+v^\va\partial_yw^\va+w^\va\partial_zw^\va, \f w^\va)_{L^2}
\end{aligned}
\end{equation}
\begin{equation}
\label{toughdetails3}
\begin{aligned}
\f T_2:=&\va^4(M_{11},\partial_x\f u^\va)_{L^2}+\va^4(M_{21},\partial_y\f u^\va)_{L^2}+\va^4(M_{31},\partial_z\f u^\va)_{L^2}+\va^4(M_{12},\partial_x\f v^\va)_{L^2}\\&+\va^4(M_{22},\partial_y\f v^\va)_{L^2}+\va^4(M_{32},\partial_z\f v^\va)_{L^2}+\va^4(M_{13},\partial_x\f w^\va)+\va^4(M_{32},\partial_y\f w^\va)+\va^4(M_{33},\partial_z\f w^\va)
\end{aligned}
\end{equation}
\begin{equation}
\label{toughdetails4}
\begin{aligned}
\f T_3:&=\sum_{(i,j)\ne (1,3), (2,3)}\va^2\nu_2 (\Delta_hQ_{ij},\f Q_{ij}^\va)_{L^2}+\sum_{(i,j)\ne (1,3), (2,3)}\va^4\nu_2 (\Delta_hQ_{i3},\f Q_{i3}^\va)_{L^2}\\&-2\va^4(\partial_tQ_{13}, \f Q_{13}^\va)_{L^2}-2\va^4(\partial_tQ_{23}, \f Q_{23}^\va)_{L^2}\\
&-\sum_{(i,j)\ne (1,3), (2,3)}(\f u^\va\partial_x Q_{ij}+u^\va\partial_x\f Q_{ij}^\va +\f v^\va\partial_y Q_{ij}+v^\va\partial_y\f Q_{ij}^\va+\f w^\va\partial_z Q_{ij}+ w^\va\partial_z\f Q_{ij}^\va, \f Q_{ij}^\va)_{L^2}\\
&-2\sum_{j=1}^2\va^2(\f u^\va\partial_x Q_{i3}+u^\va\partial_x\f Q_{i3}^\va +\f v^\va\partial_y Q_{i3}+ v^\va\partial_y\f Q_{i3}^\va+\f w^\va\partial_z Q_{i3}+w^\va\partial_z\f Q_{i3}^\va, \f Q_{i3}^\va)_{L^2}
\end{aligned}
\end{equation}
\begin{equation}
\label{toughdetails4}
\begin{aligned}
\f T_4:&=\sum_{k=1}^3\sum_{(i,j)\ne (1,3), (2,3)}(\f S^\va_k-\f S_k, \f Q^\va_{ij})_{L^2}+2\va\sum_{k=1}^3\sum_{(i,j)= (1,3)\,\,\text{or}\,\, (2,3)}(\f S^\va_k-\f S_k, \f Q^\va_{ij})_{L^2}
\end{aligned}
\end{equation}
\begin{equation}
\label{toughdetails9}
\begin{aligned}
\f T_{5}&:=(\varepsilon^4\nu_2\partial_x^2Q_{13}+\varepsilon^4 \nu_2\partial_y^2Q_{13}+\varepsilon^2 \nu_2\partial_z^2Q_{13}, \f Q_{13}^\va)_{L^2}+(\varepsilon^4 \nu_2\partial_x^2Q_{23}+\varepsilon^4\nu_2\partial_y^2Q_{23}+\varepsilon^2 \nu_2\partial_z^2Q_{23}, \f Q_{23}^\va)_{L^2}
\end{aligned}
\end{equation}
\begin{equation}
\label{toughdetails10}
\begin{aligned}
\f T_{6}&:=-a||(\f Q_{11}^\va, \f Q_{12}^\va, \f Q_{21}^\va,\f Q_{22}^\va, \f Q_{33}^\va, \va\f Q_{13}^\va, \va\f Q_{23}^\va, \va\f Q_{31}^\va, \va\f Q_{32}^\va)||_{L^2}^2\\
\end{aligned}
\end{equation}
\begin{equation}
\label{toughdetails11}
\begin{aligned}
\f T_{7}&:=\frac b3((2Q_{11}+\f Q_{11}^\va)\f Q_{11}^\va+(2Q_{12}+\f Q_{12}^\va)\f Q_{12}^\va-2(2Q_{22}+\f Q_{22}^\va)\f Q_{22}^\va-2Q_{11}\f Q_{22}^\va-2\f Q_{11}^\va Q_{22}-2\f Q_{11}^\va \f Q_{22}^\va, \f Q_{11}^\va)_{L^2}\\
&+\frac b3((2Q_{22}+\f Q_{22}^\va)\f Q_{11}^\va+(2Q_{12}+\f Q_{12}^\va)\f Q_{12}^\va-2(2Q_{11}+\f Q_{11}^\va)\f Q_{11}^\va-2Q_{11}\f Q_{22}^\va-2\f Q_{11}^\va Q_{22}-2\f Q_{11}^\va \f Q_{22}^\va, \f Q_{22}^\va)_{L^2}\\
&+2b(Q_{11}\f Q_{12}^\va+\f Q_{11}^\va Q_{12}+\f Q_{11}^\va\f Q_{12}^\va+Q_{12}\f Q_{22}^\va+\f Q_{12}^\va Q_{22}+\f Q_{12}^\va\f Q_{22}^\va ,\f Q_{12}^\va)_{L^2}\\
&+\frac {\va ^2b}3((Q_{13}+\f Q_{13}^\va)^2, \f Q_{11}^\va-2\f Q_{22}^\va)_{L^2}+\frac {\va ^2b}3( (Q_{23}+\f Q_{23}^\va)^2, -2\f Q_{11}^\va+\f Q_{22}^\va)_{L^2}+2\va^2b ((Q_{13}+\f Q_{13}^\va)(Q_{23}+\f Q_{23}^\va), \f Q_{12}^\va)_{L^2}\\
&+\va^2b( -Q_{22}Q_{13}-\f Q_{22}^\va Q_{13}-Q_{22}\f Q_{13}^\va-\f Q_{22}^\va \f Q_{13}^\va+ Q_{12}Q_{23}+\f Q_{12}^\va Q_{23}+Q_{12}\f Q_{23}^\va+\f Q_{12}^\va \f Q_{23}^\va,\f Q_{13}^\va)_{L^2}\\
&+\va^2b( Q_{12}Q_{13}+\f Q_{12}^\va Q_{13}+Q_{12}\f Q_{13}^\va+\f Q_{12}^\va \f Q_{13}^\va- Q_{11}Q_{23}-\f Q_{11}^\va Q_{23}-Q_{11}\f Q_{23}^\va-\f Q_{11}^\va \f Q_{23}^\va, \f Q_{23}^\va)_{L^2}
\end{aligned}
\end{equation}
\begin{equation}
\label{toughdetails12}
\begin{aligned}
\f T_{8}&:=-2c((Q_{11}+\f Q_{11}^\va)\text{tr}((Q^\va)^2)-Q_{11}\text{tr}(Q^2), \f Q_{11}^\va)_{L^2}-4c((Q_{12}+\f Q_{12}^\va)\text{tr}((Q^\va)^2)-Q_{12}\text{tr}(Q^2), \f Q_{12}^\va)_{L^2}\\&-2c((Q_{22}+\f Q_{22}^\va)\text{tr}((Q^\va)^2)-Q_{22}\text{tr}(Q^2), \f Q_{22}^\va)_{L^2}-c||(\va Q_{13}^\va, \va Q_{23}^\va)||^2_{L^2}\int\text{tr}((Q^\va)^2)\\
\end{aligned}
\end{equation}
\begin{equation}
\label{toughdetails13}
\begin{aligned}
\f T_{9}:=&\sum_{(i,j)\ne (1,3), (2,3)}\va^2\nu_2 (\partial_\va\Delta_hQ_{ij},\partial_\va\f Q_{ij}^\va)_{L^2}+\sum_{(i,j)\ne (1,3), (2,3)}\va^4\nu_2 (\partial_\va\Delta_hQ_{i3},\partial_\va\f Q_{i3}^\va)_{L^2}\\&-2\va^4(\partial_\va\partial_tQ_{13}, \partial_\va\f Q_{13}^\va)_{L^2}-2\va^4(\partial_\va\partial_tQ_{23}, \partial_\va\f Q_{23}^\va)_{L^2}\\
&-\sum_{(i,j)\ne (1,3), (2,3)}(\partial_\va(\f u^\va\partial_x Q_{ij}) +\partial_\va({\f v^\va\partial_y Q_{ij}})+\partial_\va(\f w^\va\partial_z Q_{ij}), \partial_\va\f Q_{ij}^\va)_{L^2}\\
&-\sum_{(i,j)\ne (1,3), (2,3)}(\partial_\va(u+\f u^\va)\partial_x\f Q_{ij}^\va +\partial_\va(v+\f v^\va)\partial_y\f Q_{ij}^\va+\partial_\va(w+\f w^\va)\partial_z\f Q_{ij}^\va, \partial_\va\f Q_{ij}^\va)_{L^2}\\
&-2\sum_{j=1}^2\va^2(\partial_\va(\f u^\va\partial_x Q_{i3})+\partial_\va(\f v^\va\partial_y Q_{i3}+\partial_\va(\f w^\va\partial_z Q_{i3}), \partial_\va\f Q_{i3}^\va)_{L^2}\\
&-2\sum_{j=1}^2\va^2(\partial_\va(\partial_\va(u+\f u^\va)\partial_x\f Q_{i3}^\va + \partial_\va(v+\f v^\va)\partial_y\f Q_{i3}^\va+\partial_\va(w+\f w^\va)\partial_z\f Q_{i3}^\va, \partial_\va\f Q_{i3}^\va)_{L^2}
\end{aligned}
\end{equation}
\begin{equation}
\label{toughdetails14}
\begin{aligned}
\f T_{10}:&=\sum_{k=1}^3\sum_{(i,j)\ne (1,3), (2,3)}(\partial_\va\f S^\va_k-\partial_\va\f S_k, \partial_\va\f Q^\va_{ij})_{L^2}+2\va\sum_{k=1}^3\sum_{(i,j)=(1,3)\,\,\text{or}\,\, (2,3)}(\partial_\va\f S^\va_k-\partial_\va\f S_k, \partial_\va\f Q^\va_{ij})_{L^2}
\end{aligned}
\end{equation}
\begin{equation}
\label{toughdetails15}
\begin{aligned}
\f T_{11}&:=(\partial_\va(\varepsilon^4\nu_2\partial_x^2Q_{13}+\varepsilon^4 \nu_2\partial_y^2Q_{13}+\varepsilon^2 \nu_2\partial_z^2Q_{13}),\partial_\va \f Q_{13}^\va)_{L^2}+(\partial_\va(\varepsilon^4 \nu_2\partial_x^2Q_{23}+\varepsilon^4\nu_2\partial_y^2Q_{23}+\varepsilon^2 \nu_2\partial_z^2Q_{23}), \partial_\va\f Q_{23}^\va)_{L^2}
\end{aligned}
\end{equation}
\begin{equation}
\label{toughdetails16}
\begin{aligned}
\f T_{12}&:=-a||\partial_\va(\f Q_{11}^\va, \f Q_{12}^\va, \f Q_{21}^\va,\f Q_{22}^\va, \f Q_{33}^\va, \va\f Q_{13}^\va, \va\f Q_{23}^\va, \va\f Q_{31}^\va, \va\f Q_{32}^\va)||_{L^2}^2\\
\end{aligned}
\end{equation}
\begin{equation}
\label{toughdetails1}
\begin{aligned}
\f T_{13}&:=\frac b3(\partial_\va((2Q_{11}+\f Q_{11}^\va)\f Q_{11}^\va+(2Q_{12}+\f Q_{12}^\va)\f Q_{12}^\va-2(2Q_{22}+\f Q_{22}^\va)\f Q_{22}^\va-2Q_{11}\f Q_{22}^\va-2\f Q_{11}^\va Q_{22}-2\f Q_{11}^\va \f Q_{22}^\va), \partial_\va\f Q_{11}^\va)_{L^2}\\
&+\frac b3(\partial_\va((2Q_{22}+\f Q_{22}^\va)\f Q_{11}^\va+(2Q_{12}+\f Q_{12}^\va)\f Q_{12}^\va-2(2Q_{11}+\f Q_{11}^\va)\f Q_{11}^\va-2Q_{11}\f Q_{22}^\va-2\f Q_{11}^\va Q_{22}-2\f Q_{11}^\va \f Q_{22}^\va), \partial_\va\f Q_{22}^\va)_{L^2}\\
&+2b(\partial_\va(Q_{11}\f Q_{12}^\va+\f Q_{11}^\va Q_{12}+\f Q_{11}^\va\f Q_{12}^\va+Q_{12}\f Q_{22}^\va+\f Q_{12}^\va Q_{22}+\f Q_{12}^\va\f Q_{22}^\va) ,\partial_\va\f Q_{12}^\va)_{L^2}\\
&+\frac {\va ^2b}3(\partial_\va(Q_{13}+\f Q_{13}^\va)^2, \partial_\va\f Q_{11}^\va-2\partial_\va\f Q_{22}^\va)_{L^2}+\frac {\va ^2b}3(\partial_\va (Q_{23}+\f Q_{23}^\va)^2, -2\partial_\va\f Q_{11}^\va+\f Q_{22}^\va)_{L^2}\\&+2\va^2b (\partial_\va((Q_{13}+\f Q_{13}^\va)(Q_{23}+\f Q_{23}^\va)), \partial_\va\f Q_{12}^\va)_{L^2}\\
&+\va^2b( \partial_\va(-Q_{22}Q_{13}-\f Q_{22}^\va Q_{13}-Q_{22}\f Q_{13}^\va-\f Q_{22}^\va \f Q_{13}^\va+ Q_{12}Q_{23}+\f Q_{12}^\va Q_{23}+Q_{12}\f Q_{23}^\va+\f Q_{12}^\va \f Q_{23}^\va),\partial_\va\f Q_{13}^\va)_{L^2}\\
&+\va^2b( \partial_\va(Q_{12}Q_{13}+\f Q_{12}^\va Q_{13}+Q_{12}\f Q_{13}^\va+\f Q_{12}^\va \f Q_{13}^\va- Q_{11}Q_{23}-\f Q_{11}^\va Q_{23}-Q_{11}\f Q_{23}^\va-\f Q_{11}^\va \f Q_{23}^\va), \partial_\va\f Q_{23}^\va)_{L^2}
\end{aligned}
\end{equation}
\begin{equation}
\label{toughdetails18}
\begin{aligned}
\f T_{14}&:=-2c(\partial_\va((Q_{11}+\f Q_{11}^\va)\text{tr}((Q^\va)^2)-Q_{11}\text{tr}(Q^2))\partial_\va, \partial_\va\f Q_{11}^\va)_{L^2}-4c(\partial_\va((Q_{12}+\f Q_{12}^\va)\text{tr}((Q^\va)^2)-Q_{12}\text{tr}(Q^2)), \partial_\va\f Q_{12}^\va)_{L^2}\\&-2c(\partial_\va((Q_{22}+\f Q_{22}^\va)\text{tr}((Q^\va)^2)-Q_{22}\text{tr}(Q^2)),\partial_\va \f Q_{22}^\va)_{L^2}-c||\partial_\va(\va Q_{13}^\va, \va Q_{23}^\va)||^2_{L^2}\int\text{tr}((Q^\va)^2)\\
\end{aligned}
\end{equation}
now we give the estimates of $\f T_0$ to $\f T_{14}$. From Cauchy-Schwartz inequality, we have
\begin{equation}
\label{toughdetails19}
\f T_0\lesssim \frac{\nu_1}{100}||(\f u^\va, \f v^\va, \va\f w^\va)||_{L^2}^2+\va^4||\Delta_h(u,v, \va w)||^2_{L^2}+\va^2||\partial_t w||_{L^2}^2
\end{equation}
remind that $\partial_xu+\partial_yv+\partial_zw=0, \quad \partial_x\f u^\va+\partial_y\f v^\va+\partial_z\f w^\va=0$, we have
\begin{equation}
\label{toughc2}
\begin{aligned}
\f T_1=&-(\f u^\va\partial_xu+\f v^\va\partial_yu+\f w^\va\partial_zu,\f u^\va)-(\f u^\va\partial_xv+\f v^\va\partial_yv+\f w^\va\partial_zv,\f v^\va)\\&
-\va^2(u\partial_xw+v\partial_yw+w\partial_zw+\f u^\va\partial_xw+\f v^\va\partial_yw+\f w^\va\partial_zw,\f w^\va)\\
\end{aligned}
\end{equation}
remind that $||(\f u^\va, \f v^\va, \va\f w^\va)||_{L^2}, ||\partial_\va (u,v,w)||_{L^\infty}\lesssim\va$. Thus
\begin{equation}
\label{toughc3}
\begin{aligned}
\mathfrak T_1&\lesssim ||(\f u^\va, \f v^\va, \va \f w^\va)||^2_{L^2}+\va^2||u\partial_xw+v\partial_yw+w\partial_zw||_{L^2}^2
\end{aligned}
\end{equation}
for the term $\f T_2$, we have
\[
\f T_2\lesssim ||\partial_\va(\f u^\va, \f v^\va, \f w^\va)||_{L^2}||\partial^2_\va(\f Q_{11}^\va, \f Q_{12}^\va, \va\f Q_{13}^\va, \va\f Q_{23}^\va)||_{L^2}
\]
for the term $\f T_3$, remind that $(\f u^\va, \f v^\va, \f w^\va)$ are divergence free, from integration by parts, for any $1\le i,j\le 3$, we have
\[
\va^2(\Delta_h Q_{ij}, \f Q^\va_{ij})_{L^2}=-\va^2(\nabla_h Q_{ij}, \nabla_h\f Q_{ij}^\va)_{L^2}, \quad (u^\va\partial_x\f Q_{ij}^\va+v^\va\partial_y\f Q_{ij}^\va+w^\va\partial_z\f Q_{ij}^\va, \f Q_{ij})_{L^2}=0
\]
this means that
\begin{equation}
\label{toughc4}
\begin{aligned}
\f T_3&\lesssim \nu_2\eta||(\va\nabla_h \f Q^\va_{11}, \va\nabla_h \f Q^\va_{12}, \va\nabla_h \f Q^\va_{12})||_{L^2}^2+\eta||(\va \f Q_{13}^\va, \va\f Q_{23}^\va)||_{L^2}^2+\eta||(\f u^\va, \f v^\va)||_{L^2}\\
&+\frac{\va^6}{\eta}||(\partial_tQ_{13},\partial_tQ_{23})||^2_{L^2} +\frac{\va^2\nu_2}{\eta}||(\nabla_h Q_{11}, \nabla_h Q_{12}, \nabla_h Q_{22}, \va\nabla_h Q_{13}, \va\nabla_hQ_{23})||_{L^2}^2\\
&+||(\partial_z Q_{11}, \partial_zQ_{12}, \partial_zQ_{22}, \va\partial_zQ_{13}, \va\partial_zQ_{23})||_{L^2}^2
\end{aligned}
\end{equation}
for the term $\f T_4$, remind that $||(Q_{11}, Q_{12}, Q_{22})||_{L^\infty}\lesssim \va$, we have
\begin{equation*}
\begin{aligned}
&\int{(\omega+\tilde\omega^\va_1)(Q_{12}+\f Q_{12}^\va)(-\f Q^\va_{11}+\f Q^\va_{22})}+\int {(\omega+\tilde\omega^\va_1)(Q_{11}-Q_{22}+\f Q_{11}^\va-\f Q_{12}^\va)\f Q_{12}^\va}\\&
-\int{\omega Q_{12}(-\f Q^\va_{11}+\f Q^\va_{22})}-\int {\omega(Q_{11}-Q_{22})\f Q^\va_{12}}\\&
=\int \tilde\omega^\va_1 [Q_{12}(-\f Q_{11}^\va+\f Q_{22}^\va)-(Q_{11}-Q_{22})\f Q_{12}^\va]\lesssim ||\va\nabla_h\f u^\va||_2^2+||(\f Q^\va_{11}, \f Q^\va_{12}, \f Q^\va_{22})||^2_{L^2}
\end{aligned}
\end{equation*}
moreover, 
\begin{equation*}
\begin{aligned}
&\int{(\va^2\partial_x(w+\f w^\va)-\partial_z(u+\f u^\va))(Q_{13}+\f Q_{13}^\va)(-\f Q^\va_{11}+\f Q^\va_{33})}\\+&\int {(\va^2\partial_x(w+\f w^\va)-\partial_z(u+\f u^\va))(Q_{11}-Q_{33}+\f Q_{11}^\va-\f Q_{33}^\va)\f Q_{13}^\va}\\
+&\int{\partial_zu Q_{13}(-\f Q^\va_{11}+\f Q^\va_{23})}+\int {\partial_zu(Q_{11}-Q_{33})\f Q^\va_{13}}\\
=&\va^2\int\partial_x(w+\f w^\va)(Q_{13}(-\f Q_{11}^\va+\f Q_{33}^\va)+(Q_{11}-Q_{33})\f Q_{13}^\va)
-\int \partial_z\f u^\va(Q_{13}(-\f Q_{11}^\va+\f Q_{33}^\va)+(Q_{11}-Q_{33})\f Q_{13}^\va)\\
\lesssim &||\va^2\partial_x\f w^\va||_{L^2}^2 +||\partial_z\f u^\va||_2^2+||(\f Q^\va_{11}, \f Q^\va_{12}, \f Q^\va_{22}, \va \f Q_{13}^\va, \va \f Q_{23}^\va)||^2_{L^2}\\
+&\va^2||(\partial_xwQ_{11}, \partial_xwQ_{12}, \partial_xwQ_{22}, \va \partial_xwQ_{13}, \partial_xwQ_{13})||^2_{L^2}
\end{aligned}
\end{equation*}
similarly for other terms, we have
\begin{equation}
\label{toughc5}
\begin{aligned}
&\mathfrak T_4\lesssim \va^2||\partial_\va(\f u^\va, \f v^\va, \f w^\va)||^2_{L^2}+||(\f Q^\va_{11}, \f Q^\va_{12}, \f Q^\va_{22}, \va \f Q_{13}^\va, \va \f Q_{23}^\va)||^2_{L^2}\\&+\va^2||(\partial_xwQ_{11}, \partial_xwQ_{12}, \partial_xwQ_{22}, \va \partial_xwQ_{13}, \partial_xwQ_{13})||^2_{L^2}
\end{aligned}\end{equation}
next, we have
\[
\f T_{5}+\nu_2||\partial_\va (\va \f Q_{13}^\va, \va \f Q_{23}^\va)||_{L^2}^2\le \nu_2||\partial_\va (\va \f Q_{13}^\va, \va \f Q_{23}^\va)||_{L^2}||\partial_\va (\va Q_{13}, \va Q_{23})||_{L^2}
\]
thus
\begin{equation}
\label{toughc9}
\begin{aligned}
\f T_{5}+\frac{\nu_2}2||\partial_\va (\va \f Q_{13}^\va, \va \f Q_{23}^\va)||_{L^2}^2\le \nu_2||\partial_\va (\va Q_{13}, \va Q_{23})||^2_{L^2}
\end{aligned}
\end{equation}
moreover,
\begin{equation}
\label{toughc10}
\begin{aligned}
 \f T_{7}&\lesssim b ||(Q_{11}, Q_{12}, Q_{22})||^2_{L^2}+b||(Q_{11}, Q_{12}, Q_{22})||^4_{L^4}\\&+b||(\f Q^\va_{11}, \f Q^\va_{12}, \f Q^\va_{22},  \va\f Q^\va_{13},  \f Q^\va_{23})||^2_{L^2}+b||(\f Q^\va_{11}, \f Q^\va_{12}, \f Q^\va_{22}, \va \f Q_{13}^\va, \va \f Q_{23}^\va)||_{L^4}^4\\
 &+b\va^4||(Q_{13}, Q_{23})||_{L^4}^4+b\va^2||(Q_{13}, Q_{23})||_{L^2}^2
 \end{aligned}
\end{equation}
for the term $\f T_{8}$, notice that
\begin{equation*}
\begin{aligned}
 &\f T_{8}+2c ||(\f Q^\va_{11}, \f Q^\va_{12}, \f Q^\va_{22}, \va\f Q_{13}^\va, \va\f Q_{23}^\va)||_{L^2}^2\int\text{tr}((Q^\va)^2)+4c||(Q_{11}\f Q_{11}^\va, Q_{12}\f Q_{12}^\va, Q_{22}\f Q_{22}^\va)||_{L^2}^2\\
&=-2c((\f Q_{11}^\va)^2+(2Q_{12}+\f Q_{12}^\va)\f Q_{12}^\va+(2Q_{22}+\f Q_{22}^\va)\f Q_{22}^\va+Q_{11}\f Q_{22}^\va+2\f Q_{11}^\va Q_{22}+2\f Q_{11}^\va \f Q_{22}^\va, Q_{11}\f Q_{11}^\va)_{L^2}\\
  &-2c((2Q_{11}+\f Q_{11}^\va)\f Q_{11}^\va+(\f Q_{12}^\va)^2+(2Q_{22}+\f Q_{22}^\va)\f Q_{22}^\va+Q_{11}\f Q_{22}^\va+2\f Q_{11}^\va Q_{22}+2\f Q_{11}^\va \f Q_{22}^\va, 2Q_{12}\f Q_{12}^\va)_{L^2}\\
   &-2c((2Q_{11}+\f Q_{11}^\va)\f Q_{11}^\va+(2Q_{12}+\f Q_{12}^\va)\f Q_{12}^\va+(\f Q_{22}^\va)^2+Q_{11}\f Q_{22}^\va+2\f Q_{11}^\va Q_{22}+2\f Q_{11}^\va \f Q_{22}^\va, Q_{22}\f Q_{22}^\va)_{L^2}\\
   &-2c\va^2((Q_{13}+\f Q^\va_{13})^2+(Q_{23}+\f Q^\va_{23})^2, Q_{11}\f Q_{11}^\va+2Q_{12}\f Q_{12}^\va+Q_{22}\f Q_{22}^\va)_{L^2}\\
 \end{aligned}
\end{equation*}
thus
\begin{equation}
\label{toughc11}
\begin{aligned}
&\f T_{8}+c||(\f Q_{11}^\va, \f Q_{12}^\va, \f Q_{22}^\va, \va\f Q_{13}^\va, \va\f Q_{23}^\va)||^4_{L^4}\\
 &\lesssim c(||(Q_{11},Q_{12},Q_{22}, \va Q_{13}, \va Q_{23})||_{L^2}^2+||(Q_{11},Q_{12},Q_{22}, \va Q_{13}, \va Q_{23})||_{L^4}^4+||(\f Q_{11}^\va, \f Q_{12}^\va, \f Q_{22}^\va, \va \f Q_{13}^\va,  \va \f Q_{23}^\va)||^2_{L^2})\\
 \end{aligned}
\end{equation}
 next, for term $\f T_{9}$, similarly as $\f T_3$, we have
 \begin{equation*}
\begin{aligned}
&\sum_{(i,j)\ne (1,3), (2,3)}\va^2\nu_2 (\partial_\va\Delta_hQ_{ij},\partial_\va\f Q_{ij}^\va)_{L^2}+\sum_{(i,j)\ne (1,3), (2,3)}\va^4\nu_2 (\partial_\va\Delta_hQ_{i3},\partial_\va\f Q_{i3}^\va)_{L^2}\\&-2\va^4(\partial_\va\partial_tQ_{13}, \partial_\va\f Q_{13}^\va)_{L^2}-2\va^4(\partial_\va\partial_tQ_{23}, \partial_\va\f Q_{23}^\va)_{L^2}\\
&\lesssim\frac{\nu_2\eta}{200}||\partial_\va^2 (\f Q_{11}^\va, \f Q_{12}^\va, \f Q_{22}^\va, \va \f Q_{13}^\va, \va\f Q_{23}^\va)||^2_{L^2}\\&+\va^4||(\Delta_h Q_{11}, \Delta_hQ_{12}, \Delta_h Q_{22}, \va \Delta _hQ_{13}, \va \Delta_h Q_{23}, \va\partial_tQ_{13}, \va\partial_tQ_{23})||_{L^2}^2\\
\end{aligned}
\end{equation*}
and remind that $||\partial_\va (Q_{11}, Q_{12}, Q_{22}, \va Q_{13}, \va Q_{23})||_{L^\infty}\lesssim \va$, we have
 \begin{equation*}
\begin{aligned}
&-\sum_{1\le i,j\le 2}(\partial_\va(\f u^\va\partial_x Q_{ij})+\partial_\va({\f v^\va\partial_y Q_{ij}})+\partial_\va(\f w^\va\partial_z Q_{ij}), \partial_\va\f Q_{ij}^\va)_{L^2}\\
&-\sum_{1\le i,j\le 2}(\partial_\va(\partial_\va(u+\f u^\va)\partial_x\f Q_{ij}^\va+\partial_\va(v+\f v^\va)\partial_y\f Q_{ij}^\va+\partial_\va(w+\f w^\va)\partial_z\f Q_{ij}^\va, \partial_\va\f Q_{ij}^\va)_{L^2}\\
&\lesssim\frac{\nu_2\eta}{200}||\partial_\va^2 (\f Q_{11}^\va, \f Q_{12}^\va, \f Q_{22}^\va, \va \f Q_{13}^\va, \va\f Q_{23}^\va)||^2_{L^2}+||\partial_\va(\f u^\va, \f v^\va, \va \f w^\va)||_{L^2}^2\\
&+||\partial_\va (\f Q_{11}^\va, \f Q_{12}^\va, \f Q_{22}^\va, \va\f Q_{13}^\va, \va\f Q_{23}^\va)||_{L^2}^2+||\partial_\va (\f Q_{11}^\va, \f Q_{12}^\va, \f Q_{22}^\va, \va\f Q_{13}^\va, \va\f Q_{23}^\va)||_{L^4}^4\\
\end{aligned}
\end{equation*}
as a result, we have
\begin{equation}
\label{toughc12}
\begin{aligned}
\f T_{9}&\lesssim\frac{\nu_2\eta}{100}||\partial_\va^2 (\f Q_{11}^\va, \f Q_{12}^\va, \f Q_{22}^\va, \va \f Q_{13}^\va, \va\f Q_{23}^\va)||^2_{L^2}\\&+\va^4||(\Delta_h Q_{11}, \Delta_hQ_{12}, \Delta_h Q_{22}, \va \Delta_h Q_{13}, \va \Delta_h Q_{23}, \va\partial_tQ_{13}, \va\partial_tQ_{23})||_{L^2}^2\\
&+||\partial_\va(\f u^\va, \f v^\va, \va \f w^\va)||_{L^2}^2+||\partial_\va (\f Q_{11}^\va, \f Q_{12}^\va, \f Q_{22}^\va, \va\f Q_{13}^\va, \va\f Q_{23}^\va)||_{L^2}^2+||\partial_\va (\f Q_{11}^\va, \f Q_{12}^\va, \f Q_{22}^\va, \va\f Q_{13}^\va, \va\f Q_{23}^\va)||_{L^4}^4\\
\end{aligned}
\end{equation}
next, for term $\f T_{10}$, remind that $||(Q_{11}, Q_{12},Q_{22}, \va Q_{13}, \va Q_{13})||_{L^\infty}\lesssim \va$, we have

\begin{equation*}
\begin{aligned}
&\int{\partial_\va((\omega+\tilde\omega^\va_1)(Q_{12}+\f Q_{12}^\va))\partial_\va(-\f Q^\va_{11}+\f Q^\va_{22})}+\int {\partial_\va((\omega+\tilde\omega^\va_1)(Q_{11}-Q_{22}+\f Q_{11}^\va-\f Q_{12}^\va))\partial_\va\f Q_{12}^\va}\\&
-\int{\partial_\va(\omega Q_{12})\partial_\va(-\f Q^\va_{11}+\f Q^\va_{22})}-\int {\partial_\va(\omega(Q_{11}-Q_{22}))\partial_\va\f Q^\va_{12}}\\
&=-\int \partial_z\omega\f Q_{12}^\va \partial_\va^2(-\f Q_{11}^\va+\f Q^\va_{22})+\int \partial_z\omega(\f Q_{11}^\va-\f Q_{22}^\va) \partial_\va^2\f Q^\va_{12}\\
&-\int \tilde\omega^\va\f Q_{12}^\va \partial_\va^2(-\f Q_{11}^\va+\f Q^\va_{22})+\int \tilde \omega^\va(\f Q_{11}^\va-\f Q_{22}^\va) \partial_\va^2\f Q^\va_{22}\\
&-\int \partial_z\tilde \omega^\va Q_{12} \partial_\va^2(-\f Q_{11}^\va+\f Q^\va_{22})+\int \partial_z\tilde\omega^\va(Q_{11}-Q_{22}) \partial_\va^2\f Q^\va_{22}\\
&\lesssim||\partial_\va^2(\f Q_{11}^\va, \f Q_{12}^\va, \f Q_{22}^\va)||^2_{L^2}+||\va\partial_h(\f u^\va, \f v^\va)||_{L^2}+||(\f Q_{11}^\va, \f Q_{12}^\va, \f Q_{22}^\va)||^2_{L^2}
\end{aligned}
\end{equation*}
and
\begin{equation*}
\begin{aligned}
&\int{\partial_\va[(\va^2\partial_x(w+\f w^\va)-\partial_z(u+\f u^\va))(Q_{13}+\f Q_{13}^\va)]\cdot\partial_\va(-\f Q^\va_{11}+\f Q^\va_{33})}\\
+&\int {\partial_\va[(\va^2\partial_x(w+\f w^\va)-\partial_z(u+\f u^\va))(Q_{11}-Q_{33}+\f Q_{11}^\va-\f Q_{33}^\va)]\partial_\va\f Q_{13}^\va}\\
+&\int{\partial_\va(\partial_zu Q_{13})\partial_\va(-\f Q^\va_{11}+\f Q^\va_{23})}+\int {\partial_\va(\partial_zu(Q_{11}-Q_{33}))\partial_va\f Q^\va_{13}}\\
=&-\va^2\int\partial_x(w+\f w^\va)(Q_{13}\partial_\va^2(-\f Q_{11}^\va+\f Q_{33}^\va)+(Q_{11}-Q_{33})\partial_\va^2\f Q_{13}^\va)\\
+&\int \partial_zu\f Q_{13}^\va \partial_\va^2(-\f Q_{11}^\va+\f Q^\va_{33})-\int \partial_zu(\f Q_{11}^\va-\f Q_{33}^\va) \partial_\va^2\f Q^\va_{13}\\+&\int \partial_z\f u^\va\f Q_{13}^\va \partial_\va^2(-\f Q_{11}^\va+\f Q^\va_{33})-\int \partial_z\f u^\va(\f Q_{11}^\va-\f Q_{33}^\va) \partial_\va^2\f Q^\va_{13}\\
+&\int \partial_z\f u^\va (Q_{13}\partial_\va^2(-\f Q_{11}^\va+\f Q_{33}^\va)+(Q_{11}-Q_{33})\partial_\va^2\f Q_{13}^\va)\\
\lesssim &||\partial_\va^2(\f Q_{11}^\va, \f Q_{12}^\va, \f Q_{22}^\va, \va\f Q_{13}^\va, \va\f Q_{23}^\va)||^2_{L^2}+||\partial_z\f u^\va||_{L^2}+||\va^2\partial_z\f w^\va||_{L^2}^2\\+&||(\f Q_{11}^\va, \f Q_{12}^\va, \f Q_{22}^\va, \va \f Q_{13}^\va, \va\f Q_{23}^\va)||^2_{L^2}+||\va^2\partial_xw Q_{13}||_{L^2}^2
\end{aligned}
\end{equation*}
after calculating all the other terms, we have
\begin{equation}
\label{toughc15}
\begin{aligned}
\f T_{10}\lesssim &||\partial_\va^2(\f Q_{11}^\va, \f Q_{12}^\va, \f Q_{22}^\va, \va\f Q_{13}^\va, \va\f Q_{23}^\va)||^2_{L^2}+||\partial_\va(\f u^\va, \f v^\va, \f w^\va)||^2_{L^2}\\+&||(\f Q_{11}^\va, \f Q_{12}^\va, \f Q_{22}^\va, \va \f Q_{13}^\va, \va\f Q_{23}^\va)||^2_{L^2}+||\va^2\partial_hw (Q_{13}, Q_{23})||_{L^2}^2
\end{aligned}
\end{equation}
next, we have
\[
\f T_{11}+\nu_2||\partial^2_\va (\va \f Q_{13}^\va, \va \f Q_{23}^\va)||_{L^2}^2\le \nu_2||\partial^2_\va (\va \f Q_{13}^\va, \va \f Q_{23}^\va)||_{L^2}||\partial^2_\va (\va Q_{13}, \va Q_{23})||_{L^2}
\]
thus
\begin{equation}
\label{toughc16}
\begin{aligned}
\f T_{11}+\frac{\nu_2}2||\partial^2_\va (\va \f Q_{13}^\va, \va \f Q_{23}^\va)||_{L^2}^2\le \nu_2||\partial^2_\va (\va Q_{13}, \va Q_{23})||^2_{L^2}
\end{aligned}
\end{equation}
%%%%%%%%%%%%%%%%%%%%%%%%%%%%%%%%%%%%%%%%%%%%%%%%%
finally, similarly as the estimate of terms $\f T_7, \f T_8$, we have

\begin{equation}
\label{toughc17}
\begin{aligned}
&\f T_{13}+\f T_{14}+\frac c2||\partial_\va((\f Q_{11}^\va)^2, (\f Q_{12}^\va)^2, (\f Q_{22}^\va)^2, (\va\f Q_{13}^\va)^2, (\va\f Q_{23}^\va)^2)||^2_{L^2}\\
 &\lesssim c||(\f Q_{11}^\va, \f Q_{12}^\va, \f Q_{22}^\va, \va \f Q_{13}^\va,  \va \f Q_{23}^\va)||^2_{L^2}\\
 &+c(||(Q_{11},Q_{12},Q_{22}, \va Q_{13}, \va Q_{23})||_{L^2}^2+||\partial_\va((Q_{11})^2,(Q_{12})^2,(Q_{22})^2, (\va Q_{13})^2, (\va Q_{23})^2)||_{L^2}^2)\\
 \end{aligned}
\end{equation}
combine the estimates above, choose $\eta$ small enough, we obtain that for $a,c$ large enough, $\frac d{dt}\f H^\va(t)\lesssim F^\va(t)$, where
\begin{equation}
\label{toughc18}
\begin{aligned}
&F^\va(t):=\va^4||\Delta_h(u,v, \va w)||^2_{L^2}+\va^2||\partial_t w||_{L^2}^2+\va^2||u\partial_xw+v\partial_yw+w\partial_zw||_{L^2}^2\\&
+\va^6||(\partial_tQ_{13},\partial_tQ_{23})||^2_{L^2} +||\partial_\va( Q_{11}, Q_{12}, Q_{22}, \va Q_{13}, \va Q_{23})||_{L^2}^2\\&+\va^2||(\partial_xwQ_{11}, \partial_xwQ_{12}, \partial_xwQ_{22}, \va \partial_xwQ_{13}, \partial_xwQ_{13})||^2_{L^2}+||\partial_\va (\va Q_{13}, \va Q_{23})||^2_{L^2}\\&+||(Q_{11},Q_{12},Q_{22}, \va Q_{13}, \va Q_{23})||_{L^2}^2+||(Q_{11},Q_{12},Q_{22}, \va Q_{13}, \va Q_{23})||_{L^4}^4\\&+\va^4||(\Delta_h Q_{11}, \Delta_hQ_{12}, \Delta_h Q_{22}, \va \Delta_h Q_{13}, \va \Delta_h Q_{23}, \va\partial_tQ_{13}, \va\partial_tQ_{23})||_{L^2}^2+||\va^2\partial_hw (Q_{13}, Q_{23})||_{L^2}^2\\
&+||\partial^2_\va (\va Q_{13}, \va Q_{23})||^2_{L^2}+||\partial_\va((Q_{11})^2,(Q_{12})^2,(Q_{22})^2, (\va Q_{13})^2, (\va Q_{23})^2)||_{L^2}^2
\end{aligned}
\end{equation}
so there exists a constant $\f C_l$, such that
\[
\f H^\va(t)\le \f H^\va(0)+\f C_l\int_0^t F^\va(s)ds
\]
remind that we suppose that $\int_0^\infty F^\va(s)ds\lesssim \va^2$, we obtain that $T^*=\infty$. This finishes the proof of Theorem \ref{thmconvergence1}.
\end{proof}

\begin{appendix}
\section{Scaling calculations}\label{apendixsec:scaling}
 We have
\begin{equation}
\label{tensor3}
\nabla Q\odot\nabla Q=\left(
\begin{array}{ccc}\sum_{ij}(\partial_xQ_{ij})^2& \sum_{ij}\partial_xQ_{ij}\partial_yQ_{ij}& \sum_{ij}\partial_xQ_{ij}\partial_zQ_{ij}\\
 \sum_{ij}\partial_xQ_{ij}\partial_yQ_{ij}& \sum_{ij}(\partial_yQ_{ij})^2& \sum_{ij}\partial_yQ_{ij}\partial_zQ_{ij}\\
 \sum_{ij}\partial_xQ_{ij}\partial_zQ_{ij}& \sum_{ij}\partial_yQ_{ij}\partial_zQ_{ij}&\sum_{ij}(\partial_zQ_{ij})^2
\end{array}
\right)
\end{equation}
Furthermore $(\Delta Q\cdot Q-Q\cdot\Delta Q)_{ij}=-(\Delta Q\cdot Q-Q\cdot\Delta Q)_{ji}$, and
\[
(\Delta Q\cdot Q-Q\cdot\Delta Q)_{12}=-\Delta Q_{12}Q_{11}+\Delta Q_{12}Q_{22}+\Delta(Q_{11}-Q_{22})Q_{12}-\Delta Q_{23}Q_{13}+\Delta Q_{13}Q_{23}
\]
\[
(\Delta Q\cdot Q-Q\cdot\Delta Q)_{13}=-2\Delta Q_{13}Q_{11}-\Delta Q_{13}Q_{22}-\Delta Q_{23}Q_{12}+\Delta (2Q_{11}+Q_{22})Q_{13}+\Delta Q_{12}Q_{23}
\]
\[
(\Delta Q\cdot Q-Q\cdot\Delta Q)_{23}=-\Delta Q_{23}Q_{11}-2\Delta Q_{23}Q_{22}-\Delta Q_{13}Q_{12}+\Delta Q_{12}Q_{13}+\Delta (Q_{11}+2Q_{22})Q_{23}
\]
Define the matrix 
\begin{equation}\label{def:matrixM}
M_{ij}:=(\nabla Q \odot\nabla Q+(\Delta Q)\cdot Q-Q\cdot\Delta Q)_{ij}, 1\le i,j\le 3.
\end{equation}

Furthermore,
\[
\mbox{tr}(Q^2)=2(Q^2_{11,\varepsilon}+Q^2_{12,\varepsilon}+Q^2_{22,\varepsilon}+\varepsilon^2Q^2_{13,\varepsilon}+\varepsilon^2Q^2_{23,\varepsilon}+Q_{11,\varepsilon}Q_{22,\varepsilon})
\]
denote 
\begin{equation}\label{def:omegavareps}
\omega^\varepsilon_0:=\partial_xv^\varepsilon-\partial_yu^\varepsilon,\quad \omega^\varepsilon_1:=\varepsilon^2\partial_xw^\varepsilon-\partial_zu^\varepsilon,\quad \omega^\varepsilon_2:=\varepsilon^2\partial_yw^\varepsilon-\partial_zv^\varepsilon
\end{equation}
then from direct calculations,
\begin{equation}
\label{3omega2}
\Omega^\va=\frac 12\left(
\begin{array}{ccc}
0 & \omega_0^\va&\omega^\va_1/\varepsilon\\
-\omega_0^\va& 0&\omega^\va_2/\varepsilon\\
-\omega^\va_1/\varepsilon&-\omega^\va_2/\varepsilon&0
\end{array}
\right)
\end{equation}

\end{appendix}

%%%%%%%%%%%%%%%%%%%%%%%%%%%%%%%%%%%%%%%%%%%%%%%%%%%%%%%%%%
%%%%%%%%%%%%%%%%%%%%%%%%%%%%%%%%%%%%%%%%%%%%%%%%%%%%%%%%%%%%%%

\section*{\bf Acknowledgments.}A. Zarnescu and  X.Li have been partially supported by the Basque Government through the BERC 2022- 2025 program and by the Spanish State Research Agency through BCAM Severo Ochoa excellence accreditation SEV-2017-0718.  X. Li has been also supported by Unversit\'e de Toulouse III Paul Sabatier and  by the LTC-Transmath project funded by Fundaci\'on Euskampus. A.Zarnescu has also been supported  through project PID2023-146764NB-I00 funded by Agencia Estatal de Investigaci\'on (PID2023-146764NB-I00/ AEI / 10.13039/501100011033) and also partially supported by a grant of the Ministry of Research, Innovation and Digitization,
CNCS - UEFISCDI, project number PN-III-P4-PCE-2021-0921, within PNCDI III. M.Paicu has been
supported by Universit\'e de Bordeaux.\\

\begin{footnotesize}

\end{footnotesize}
F.De Anna: Institute of Mathematics, University of Wurzburg ,Wurzburg, Germany.\\
E-mail address: francesco.deanna@uni-wuerzburg.de\\
X.Li: Department of Mathematics and Geosciences, University of Trieste, via Valerio 12/1
Trieste, 34127 Italy. \\
E-mail address: xingyuli92@gmail.com\\
M.Paicu: Universit\'e Bordeaux, Institut de Math\'ematiques de Bordeaux, F-33405 Talence Cedex, France.\\
Email address: mpaicu@math.u-bordeaux1.fr\\
A.Zarnescu: BCAM, Basque Center for Applied Mathematics, Mazarredo 14, E48009 Bilbao, Bizkaia, Spain\\
\&IKERBASQUE, Basque Foundation for Science, Maria Diaz de Haro 3, 48013, Bilbao, Bizkaia, Spain\\
\&Simion Stoilow Institute of Mathematics of the Romanian Academy, P.O. Box 1-764, RO-014700 Bucharest, Romania.\\
Email address: azarnescu@bcamath.org
\end{document}